\documentclass[a4, reqno]{amsart}
\usepackage{amsmath}
\usepackage{amsthm}
\usepackage{amssymb}
\usepackage{amscd}

\usepackage{hyperref}

\usepackage[margin=3.75cm]{geometry}

\theoremstyle{plain}
\newtheorem{theorem}{Theorem}[section]
\newtheorem{proposition}[theorem]{Proposition}

\newtheorem{lemma}[theorem]{Lemma}

\theoremstyle{definition}
\newtheorem{definition}[theorem]{Definition}

\newtheorem{assumption}[theorem]{Assumption}

\newtheorem{conventions}[theorem]{Conventions}

\theoremstyle{remark}
\newtheorem{remark}[theorem]{Remark}
\newtheorem{example}[theorem]{Example}

\numberwithin{equation}{section}

\newcommand{\bv}{\mathbf{v}}

\newcommand{\Moduli}[4]{{\mathcal{M}_{#3}(#1, #2, #4)}}
\newcommand{\FormalHilb}[2]{{\mathcal{H}^{#1}_{#2}}}

\title[Moduli of $1$-dimensional sheaves on log surfaces and Hilbert schemes]{Moduli spaces of one dimensional sheaves on log surfaces and Hilbert schemes}

\author{Nobuyoshi Takahashi}
\address{
Department of Mathematics, 
Graduate School of Advanced Science and Engineering, 
Hiroshima University, 
1-3-1 Kagamiyama, Higashi-Hiroshima, 
739-8526 JAPAN}
\email{tkhsnbys@hiroshima-u.ac.jp}

\subjclass[2020]{Primary 14B05; Secondary 14H40; 14H60; 14J42}

\keywords{Log surface; Moduli space; Compactified Jacobian; Hilbert scheme; Symplectic singularity}

\begin{document}

\maketitle

\begin{abstract}
Let $X$ be a smooth projective rational surface, 
$D\subset X$ an effective anticanonical curve, 
$\beta$ a curve class on $X$ 
and $\mathfrak{d}=\sum w_iP_i$ an effective divisor on $D_{\mathrm{sm}}$. 
We consider the moduli space $\Moduli{X}{D}{\beta}{\mathfrak{d}}$ of sheaves on $X$ 
which are direct images of rank-$1$ torsion-free sheaves on integral curves $C$ in $\beta$ 
such that $C|_D=\mathfrak{d}$, 
and show that each point of $\Moduli{X}{D}{\beta}{\mathfrak{d}}$ 
is smooth over a point 
in the product of the Hilbert schemes of surface singularities of types $A_{w_i-1}$. 
Hence, $\Moduli{X}{D}{\beta}{\mathfrak{d}}$ 
has symplectic singularities and admits a unique symplectic resolution. 
\end{abstract}

\section{Introduction}

For a proper algebraic variety $X$, 
the moduli spaces of coherent sheaves on $X$ with one dimensional supports 
have applications in a number of topics, 
such as enumeration of curves on $X$. 
In the case where $X$ is a K3 surface, 
the moduli spaces of simple sheaves are symplectic manifolds (\cite{Mukai1984}), 
and geometric properties of their compactifications 
have been studied in many works 
(e.\,g. \cite{OG1999}, \cite{KLS2006}, \cite{AS2018}, \cite{PR2023}, \cite{CKLR2024}). 
One important question asked there is whether they 
have symplectic singularities and admit symplectic resolutions. 

Let $(X, D)$ be a logarithmic pair consisting of 
a smooth surface $X$ over $\mathbb{C}$ 
and an effective divisor $D$ on $X$, 
$\beta$ a curve class on $X$ 
and $\mathfrak{d}$ an effective divisor on $D_{\mathrm{sm}}$.  
Then we consider the moduli space $\Moduli{X}{D}{\beta}{\mathfrak{d}}$ 
of sheaves on $X$ which are direct images of rank-$1$ torsion-free sheaves 
on integral curves $C$ in $\beta$ satisfying $C|_D=\mathfrak{d}$. 
This is a partial compactification of a space considered in \cite{BiswasGomez2020}, 
and such moduli spaces would be of interest also in the context of 
Hitchin pairs (\cite{Nitsure1991}, \cite{Bottacin1995a}, \cite{Markman1994}). 

The main question we ask here is whether 
$\Moduli{X}{D}{\beta}{\mathfrak{d}}$  has symplectic singularities 
and admits symplectic resolutions. 
We are especially interested in the case 
where $X$ is a smooth projective rational surface and $D$ is an anticanonical curve. 
In this case, the pair can be thought of as a logarithmic analogue of a K3 surface. 

Our moduli space can be regarded as the relative compactified Jacobian 
of a family $\mathcal{C}/\Lambda$ of curves. 
Hilbert schemes and Abel maps have been major tools 
in the study of the geometry of compactified Jacobians 
(see e.\,g. \cite{DSouza1979}, \cite{AK1980}, \cite{MRV2017} and \cite{Shende2012}). 
In this paper, we use twisted Abel maps from \cite{MRV2017} 
to give an isomorphism of the completions of $\Moduli{X}{D}{\beta}{\mathfrak{d}}$  
and the relative Hilbert scheme 
$\mathrm{Hilb}^g(\mathcal{C}/\Lambda)$. 
Since the latter can be reduced to the case of a subscheme $Z$ supported at a point, 
we consider in this introduction the relative formal Hilbert scheme 
$\FormalHilb{\hat{\mathcal{C}}/\hat{\Lambda}}{Z}$ 
where $\hat{\mathcal{C}}$ and $\hat{\Lambda}$ are 
completions at one point. 

What is special about our families is that 
the curves meet the boundary curve with constant multiplicity $w$. 
Note that, if we take a general such family 
with a fixed singular curve as the central fiber, 
its total space is formally isomorphic 
to the product of the surface singularity $S$ of type $A_{w-1}$ 
and a smooth space (Example \ref{ex_codim2}). 
In this paper, we observe an intriguing phenomenon 
that $\FormalHilb{\hat{\mathcal{C}}/\hat{\Lambda}}{Z}$ can be related, 
by an elementary ``trick,'' 
to the formal Hilbert scheme $\FormalHilb{S}{\tilde{Z}}$ 
for a subscheme $\tilde{Z}\subset S$. 
Under a certain nondegeneracy condition (Definition \ref{def_Phi_nondegeneracy}), 
we see that 
$\FormalHilb{\hat{\mathcal{C}}/\hat{\Lambda}}{Z}$ is smooth 
over $\FormalHilb{S}{\tilde{Z}}$
(Theorem \ref{thm_corr}). 
We can somewhat relax the assumption 
by considering coordinate changes of the ambient space 
depending on the parameter of the curves 
(Proposition \ref{prop_coord_change}). 

The nondegeneracy condition holds 
in the case of a log Calabi-Yau pair $(X, D)$ with $X$ rational, 
and we have the following. 

\begin{theorem}[=Theorem \ref{thm_main}]
Let $X$ be a smooth projective rational surface, 
$D\subset X$ an anticanonical curve, 
$\beta\in H^2(X, \mathbb{Z})$ a curve class of arithmetic genus $g$ 
and $\mathfrak{d}=\sum_{i=1}^k w_iP_i$ an effective divisor on $D_{\mathrm{sm}}$ 
such that $\beta|_D\sim \mathfrak{d}$. 

Then each point of $\Moduli{X}{D}{\beta}{\mathfrak{d}}$ 
is formally isomorphic to a point 
of $\prod_{i=1}^k \mathrm{Hilb}^{d_i}(S_{w_i-1})\times \mathbb{A}^{2d}$, 
where $S_n$ is a surface singularity of type $A_n$, 
$d_i\leq \delta(C, P_i)$ and $d=g-\sum_{i=1}^k d_i$. 
\end{theorem}
Since $\mathrm{Hilb}^d(S_n)$ has symplectic singularities 
and admits a unique symplectic resolution by \cite{CY2023}, 
the same holds for $\Moduli{X}{D}{\beta}{\mathfrak{d}}$. 

The moduli space $\Moduli{X}{D}{\beta}{\mathfrak{d}}$ 
is a partial compactification of a symplectic leaf in the moduli space 
considered in \cite{BiswasGomez2020}, 
and it adds a codimension one set corresponding to sheaves 
with singular support curves. 
To compactify $\Moduli{X}{D}{\beta}{\mathfrak{d}}$, 
it is necessary to include sheaves whose support curves have 
components in common with $D$. 
One approach is to consider sheaves on degenerations of $X$ 
(\cite{MPT2010}, \cite{LW2015}), 
and it seems reasonable to expect that the compactified moduli space 
has a log-symplectic structure, as in \cite{Das2022}.

\smallbreak
This paper is organized as follows. 
In Section 2, we introduce moduli spaces of one dimensional sheaves 
with conditions on the intersections with the boundary. 
In Section 3, we explain how twisted Abel maps give 
local isomorphisms between these spaces and relative Hilbert schemes 
of families of curves. 
In Section 4, we define a morphism from such a relative Hilbert scheme 
to a Hilbert scheme of a surface singularity of type $A_n$, 
and show that it is a local isomorphism under a certain nondegeneracy condition. 
Finally, in Section 5, 
we prove the main theorem. 

\begin{conventions}
When a point of a moduli space represents a geometric or an algebraic object, 
say $\mathcal{F}$, 
we will sometimes write $[\mathcal{F}]$ for this point. 

If $\Lambda$ is the (formal) spectrum of a local $\mathbb{C}$-algebra 
with residue field $\mathbb{C}$, 
its closed point is denoted by $0$. 
Families of curves over $\Lambda$ and functions on a family over $\Lambda$ 
are symbolically written in the form 
$\mathcal{C}=\{C_\lambda\}_{\lambda\in\Lambda}$ and 
$\mathcal{E}=\{E_\lambda(x, y)\}_{\lambda\in\Lambda}$, 
and $C_0$ and $E_0(x, y)$ are their fibers over $0$. 

If $Y$ is a closed (formal) subscheme of $X$, 
its ideal sheaf is denoted by $I_{Y\subset X}$, 
or $I_Y$ when the ambient space is obvious from the context. 
\end{conventions}

\section{moduli spaces of one dimensional sheaves}

The moduli spaces of simple sheaves on a K3 surface 
are equipped with natural symplectic structures by \cite{Mukai1984}. 
In the log Calabi-Yau setting, by choosing an anticanonical form, 
we obtain a Poisson structure and its symplectic leaves as follows 
(\cite{Tyurin1988}, \cite{Bottacin1995a}, \cite{Bottacin1995b} and \cite{BiswasGomez2020}). 

Let $X$ be a nonsingular projective surface, 
$D\subset X$ an effective divisor, 
$\beta\in H^2(X)$ a class of a curve with $w:=\beta.D>0$. 

\begin{definition}
We define the following parameter spaces: 
\begin{eqnarray*}
\Lambda_\beta & := & \{C: \hbox{$C$ is an integral curve in $X$, 
$C\not\subseteq D$, $[C]=\beta$}\}, \\
\mathcal{M}_\beta & := & \{\hbox{$\mathcal{F}$}: 
\hbox{$\mathcal{F}$ is a coherent sheaf on $X$, 
torsion-free of rank $1$ on some $C\in \Lambda_\beta$}\}, \\
\mathcal{M}^\circ_\beta & := &
\{
\hbox{$\mathcal{F}$}: 
\hbox{$\mathcal{F}$ is a coherent sheaf on $X$, 
invertible on some smooth $C\in \Lambda_\beta$}
\}. 
\end{eqnarray*}
To be precise, 
$\mathcal{M}_\beta$ represents the sheafification 
of the functor on $\mathbb{C}$-schemes 
\[
\ S\mapsto 
\left.\left\{
\mathcal{F} \left\vert
\begin{array}{l}
\hbox{$\mathcal{F}$ is a coherent sheaf on $X\times S$ flat over $S$ such that,} \\
\hbox{for any geometric point $t$ of $S$, 
$\mathcal{F}_t$ is a torsion-free sheaf} \\
\hbox{of rank $1$ on some $C_t$ belonging to $\Lambda_\beta$}
\end{array}
\right.\right\}\right/
\cong, 
\]
and so on. 

There is a natural morphism 
\[
\mathcal{M}_\beta\to \Lambda_\beta;\ 
[\mathcal{F}]\mapsto [\mathrm{Supp}(\mathcal{F})] 
\]
given by the operation of taking the $0$-th Fitting ideal (\cite[\S2]{BiswasGomez2020}). 
We also have a morphism 
\[
\Lambda_\beta\to \mathrm{Sym}^w(D);\ [C]\mapsto [C.D], 
\]
where $\mathrm{Sym}^w(D)$ is identified with the set of effective $0$-cycles 
of degree $w$ on $D$. 
Let the composite be 
\[
\varphi_{X, D, \beta}: \mathcal{M}_\beta\to\mathrm{Sym}^w(D); 
\ [\mathcal{F}]\mapsto [\mathrm{Supp}(\mathcal{F}).D]. 
\]
\end{definition}

\begin{theorem}[{\cite{Tyurin1988}, \cite{Bottacin1995b}}]
A section $s\in H^0(X, \omega_X^{-1})$ 
determines a Poisson structure on $\mathcal{M}_\beta$ by 
\[
T^*_{[\mathcal{F}]}\mathcal{M}_\beta=
\mathrm{Ext}^1_{\mathcal{O}_X}(\mathcal{F}, \mathcal{F}\otimes\omega_X)
\overset{s}{\to} 
\mathrm{Ext}^1_{\mathcal{O}_X}(\mathcal{F}, \mathcal{F})
=T_{[\mathcal{F}]}\mathcal{M}_\beta. 
\]
\end{theorem}

\begin{theorem}[{\cite{BiswasGomez2020}, cf. \cite[Theorem 5.5]{Beauville1990}, \cite[Theorem 4.7.5]{Bottacin1995a}}]
Let $D$ be the divisor of a nonzero section $s\in H^0(X, \omega_X^{-1})$. 
Then the fibers of 
$\varphi_{X, D, \beta}|_{\mathcal{M}^\circ_\beta}: 
\mathcal{M}^\circ_\beta\to \mathrm{Sym}^w(D)$
are the symplectic leaves for the Poisson structure 
on $\mathcal{M}^\circ_\beta$ associated to $s$. 
\end{theorem}

We are going to study partial compactifications of these fibers. 
\begin{assumption}\label{assumption}
We consider the following conditions. 
\begin{itemize}
\item
$X$ is a smooth projective rational surface and $D\in |-K_X|$. 
\item
$\beta\in H^2(X, \mathbb{Z})\cong \mathrm{Pic}(X)$ is the class of a curve. 
\item
$[\mathfrak{d}]\in\mathrm{Sym}^w(D)$ is an effective divisor of degree $w$ 
on $D_{\mathrm{sm}}$, where $w=\beta.D$, 
and $\mathfrak{d}\sim \beta|_D$. 
\end{itemize}
\end{assumption}
Note that, if $X$ is regular ($h^1(\mathcal{O}_X)=0$) and $|-K_X|$ 
contains a nonzero divisor, then $X$ is rational. 

\begin{definition}
Let $X, D, \beta, \mathfrak{d}$ be as in Assumption \ref{assumption}. 
We define $\Moduli{X}{D}{\beta}{\mathfrak{d}}$ to be 
$\varphi_{X, D, \beta}^{-1}([\mathfrak{d}])$. 
\end{definition}

\begin{remark}
In \cite{CGKT2021b}, a moduli space 
$\mathcal{MMI}_\beta^P(X, D)$ was introduced, 
and was applied to the calculation of 
multiplicities of $\mathbb{A}^1$-curves ({\cite[Proposition 1.8 (3)]{CGKT2021a}}). 
This can be seen as an open subset of 
$\Moduli{X}{D}{\beta}{wP}$ 
parametrizing sheaves on $C\subset X$ invertible at $P$. 
\end{remark}

The main object of this paper is 
to study the singularities of $\Moduli{X}{D}{\beta}{\mathfrak{d}}$. 
By the following lemma, 
it can be regarded as the relative compactified Jacobian 
over a (non-proper) $p_a(\beta)$-dimensional base. 

\begin{lemma}\label{lem_linearsystem}
Under Assumption \ref{assumption}, 
the following hold. 
\begin{enumerate}
\item
We have $H^0(\mathcal{O}_D)\cong \mathbb{C}$, 
and hence if $L$ is an invertible sheaf on $D$ with $L\cong\mathcal{O}_D(\mathfrak{d})$, 
then $\mathfrak{d}$ is defined by a section of $L$ 
unique up to scalar multiplications. 
\item
One can identify 
$\Moduli{X}{D}{\beta}{\mathfrak{d}}$ 
with the relative compactified Jacobian over 
\begin{equation}
\Lambda_{\beta, \mathfrak{d}}:=
\{C\in |\beta| :  
\hbox{$C$ is integral, 
$C\not\subseteq D$ and 
$C|_D = \mathfrak{d}$}
\}, \label{eq_lambda_b_d}
\end{equation}
which is of dimension $p_a(\beta)$. 
\end{enumerate}
\end{lemma}
\begin{proof}
(1)
The first statement follows from the short exact sequence 
$0\to \mathcal{O}_X(-D)\to \mathcal{O}_X\to\mathcal{O}_D\to 0$ 
and $h^1(\mathcal{O}_X(-D))=h^1(\mathcal{O}_X(K_X+D))=h^1(\mathcal{O}_X)=0$. 
Then the second statement follows from the exact sequence 
$0\to L(-\mathfrak{d})\cong\mathcal{O}_D\to L\to L|_{\mathfrak{d}}\to 0$.

\medbreak
(2)
Set theoretically, the identification is obvious. 
This is also scheme-theoretically true 
by the existence of the natural morphism $\mathcal{M}_\beta\to\Lambda_\beta$, 
see \cite[\S2]{BiswasGomez2020}. 

If $C$ is an element of $\Lambda_{\beta, \mathfrak{d}}$, 
let us consider the long exact sequence 
associated to the sequence 
$0\to \mathcal{O}_X(C-D)\to \mathcal{O}_X(C)\to \mathcal{O}_X(C)|_D\to 0$. 
We see that 
$h^1(\mathcal{O}_X(C-D))=h^1(\mathcal{O}_X(K_X+D-C))=h^1(\mathcal{O}_X(-C))$ 
which is $0$ 
since $C$ is integral and $h^1(\mathcal{O}_X)=0$. 
Thus 
\[
0\to H^0(\mathcal{O}_X(C-D)) \to H^0(\mathcal{O}_X(C))
\overset{\pi}{\to} H^0(\mathcal{O}_X(C)|_D)\to 0
\]
is exact. 
On the other hand, $C|_D$ is defined by a section $\sigma\in H^0(\mathcal{O}_X(C)|_D)$ 
which is unique up to scalar multiplications by (1). 
Then $\Lambda_{\beta, \mathfrak{d}}$ is the projectivization of 
$\pi^{-1}(\mathbb{C}\sigma)$. 
We also have $h^2(\mathcal{O}_X(C-D))=h^0(\mathcal{O}_X(-C))=0$, 
and by Riemann-Roch formula 
\[
\dim \Lambda_{\beta, \mathfrak{d}} = h^0(\mathcal{O}_X(C-D)) = 
\frac{1}{2}(C-D).(C-D-K_X)+\chi(\mathcal{O}_X)=
\frac{1}{2}(C+K_X).C+1=
p_a(C). 
\]
\end{proof}

\section{Hilbert schemes and twisted Abel maps}

In this section, we recall twisted Abel maps from \cite{MRV2017} 
which provide isomorphisms between compactified Jacobians and Hilbert schemes. 
We also give definitions 
concerning families of local curves 
with fixed orders of intersection with the boundary. 

First we fix notations related to local Hilbert functors, 
following \cite{Sernesi2006}. 
\begin{definition}
\begin{enumerate}
\item
Let $R$ be a local noetherian $\mathbb{C}$-algebra with residue field $\mathbb{C}$. 
Then $\mathcal{A}_R$ 
denotes the category of local artinian 
$R$-algebras with residue field $\mathbb{C}$, 
with local $R$-homomorphisms as morphisms. 
If $R$ is complete, let $\hat{\mathcal{A}}_R$ 
be the category of complete local noetherian $R$-algebras 
with residue field $\mathbb{C}$. 
The closed point of the (formal) spectrum of a local ring 
will be denoted by $0$. 
\item
Let $R$ be a (complete) local noetherian $\mathbb{C}$-algebra 
with residue field $\mathbb{C}$ 
and $\Lambda$ its (formal) spectrum. 
For a (formal) scheme $X$ over $\Lambda$ 
and a closed subscheme $Z$ of $X_0$, 
let $H^{X/\Lambda}_Z: \mathcal{A}_R\to(\mathrm{Set})$ 
be the local Hilbert functor (\cite[\S3.2]{Sernesi2006}) 
defined as follows: 
\[
H^{X/\Lambda}_Z(A):=\{\hbox{closed subscheme $\tilde{Z}\subset X\otimes_R A$ 
flat over $A$ with $\tilde{Z}_0=Z$} \}. 
\]
If $R=\mathbb{C}$, we write $H^X_Z = H^{X/\mathrm{Spec}\ \mathbb{C}}_Z$. 
\end{enumerate}
\end{definition}

\begin{definition}
Let $\FormalHilb{X/\Lambda}{Z}$ (resp. $\FormalHilb{X}{Z}$) 
denote the formal spectrum 
of the complete local ring that pro-represents $H^{X/\Lambda}_Z$ (resp. $H^{X}_Z$). 
This is naturally isomorphic to the completion of the Hilbert (formal) scheme 
of $X/\Lambda$ 
(or of $\hat{X}/\hat{\Lambda}$ where 
$\hat{X}$ and $\hat{\Lambda}$ are completions along $Z$ and $0$) 
at $[Z]$. 

If $Z$ is defined by an ideal $I$, 
we write $\FormalHilb{X}{I}=\FormalHilb{X}{Z}$. 
\end{definition}

Let us introduce the twisted Abel maps as in \cite[\S2]{MRV2017} 
(see also \cite[Definition 5.16]{AK1980}). 
For simplicity, we will mostly work with formal neighborhoods of 
points in the moduli spaces in what follows, 
although we could work a little more globally. 

\begin{definition}
Let $S$ be the formal spectrum of an object of $\hat{\mathcal{A}}_{\mathbb{C}}$ 
and $\mathcal{C}\to S$ a proper flat family of integral curves. 
We denote the compactified Jacobian of $\mathcal{C}/S$ 
by $\bar{J}(\mathcal{C}/S)$. 
For a positive integer $d$ and an invertible sheaf $\mathcal{L}$ on $\mathcal{C}$, 
we define the \emph{$\mathcal{L}$-twisted Abel map} 
\[
A^d_{\mathcal{L}}: \mathrm{Hilb}^d(\mathcal{C}/S) \to \bar{J}(\mathcal{C}/S) 
\]
as the morphism representing the natural transform of functors 
which maps a flat family $\mathcal{Z}\subset \mathcal{C}\times_S T$ 
of subschemes over $T$ 
to 
$I_{\mathcal{Z}\subset \mathcal{C}\times_S T}
\otimes (\mathcal{L}\otimes_{\mathcal{O}_S} \mathcal{O}_T)$. 

If $[Z]\in \mathrm{Hilb}^d(\mathcal{C}_0)$ 
is a $\mathbb{C}$-valued point and 
$[\mathcal{F}]=A^d_{\mathcal{L}}([Z])$, 
let 
\[
A_{\mathcal{L}, Z}: \FormalHilb{\mathcal{C}/S}{Z} \to 
\bar{J}(\mathcal{C}/S)^\wedge_{[\mathcal{F}]} 
\]
be the induced morphism, where 
$\bar{J}(\mathcal{C}/S)^\wedge_{[\mathcal{F}]}$ 
is the completion at $[\mathcal{F}]$. 
\end{definition}

\begin{proposition}[{cf. \cite[Proposition 2.5]{MRV2017}}]\label{prop_twisted_abel_map}
Let $S$ be the formal spectrum of an object of $\hat{\mathcal{A}}_{\mathbb{C}}$, 
and $\mathcal{C}\to S$ a proper flat family of integral Gorenstein curves 
of arithmetic genus $g$, with central fiber $C$. 
Let $\mathcal{F}$ be a torsion-free sheaf of rank $1$ on $C$. 
Then there exists a line bundle $\mathcal{L}$ on $\mathcal{C}$ 
such that 
$(A^g_{\mathcal{L}})^{-1}([\mathcal{F}])=\{[Z]\}$ for a $0$-dimensional 
subscheme $Z$ of $C$ 
and 
$A_{\mathcal{L}, Z}: \FormalHilb{\mathcal{C}/S}{Z} \to 
\bar{J}(\mathcal{C}/S)^\wedge_{[\mathcal{F}]}$ 
is an isomorphism. 

More precisely, the following hold. 

(1)
There exists a $0$-dimensional subscheme $Z_0$ of $C$ 
satisfying the following conditions. 
\begin{itemize}
\item
$I_{Z_0}$ contains the conductor ideal 
$\mathcal{A}nn (\nu_* \mathcal{O}_{\tilde{C}}/\mathcal{O}_C)$, 
where $\nu: \tilde{C}\to C$ is the normalization. 
\item
$\mathcal{F}_P\cong (I_{Z_0})_P$ for any $P\in C$, 
or equivalently $\mathcal{F}\cong I_{Z_0}\otimes L_0$ 
for an invertible sheaf $L_0$. 
\end{itemize}

(2)
Assume that $Z_0$ satisfies the conditions in (1) 
and let $d$ be the length of $Z_0$. 
If we set $Z=Z_0\cup \{R_1, \dots, R_{g-d}\}$ 
for general points $R_1, \dots, R_{g-d}\in C_{\mathrm{sm}}$, 
then 
$\mathrm{Hom}_{\mathcal{O}_C}(I_Z, \mathcal{O}_C)\cong\mathbb{C}$ 
and 
$\mathrm{Ext}^1_{\mathcal{O}_C}(I_Z, \mathcal{O}_C)=0$ 
hold. 

(3) 
If $Z$ is a $0$-dimensional subscheme of $C$ satisfying 
$(I_Z)_P\cong \mathcal{F}_P$ for any $P\in C$, 
$\mathrm{Hom}_{\mathcal{O}_C}(I_Z, \mathcal{O}_C)\cong\mathbb{C}$ 
and 
$\mathrm{Ext}^1_{\mathcal{O}_C}(I_Z, \mathcal{O}_C)=0$, 
then there exists a line bundle $\mathcal{L}$ on $\mathcal{C}$ 
such that 
$(A^g_{\mathcal{L}})^{-1}([\mathcal{F}])=\{[Z]\}$, 
and for any such $\mathcal{L}$, 
$A_{\mathcal{L}, Z}: \FormalHilb{\mathcal{C}/S}{Z} \to 
\bar{J}(\mathcal{C}/S)^\wedge_{[\mathcal{F}]}$ 
is an isomorphism. 
\end{proposition}

\begin{proof}
(1)
For a point $P\in C$, 
let $R=\mathcal{O}_{C, P}$ and 
$\tilde{R}$ the normalization of $R$. 
Since $R$ is Gorenstein, 
the assignment $M\mapsto M^*:=\mathrm{Hom}_R(M, R)$ 
is involutory up to a natural isomorphism for finitely generated torsion free modules 
(\cite[Lemma 1.1]{Hartshorne1986}), 
and induces an involution on the set of fractional ideals of $R$. 
The $R$-module $(\mathcal{F}_P)^*$ 
is torsion-free of rank-$1$, 
and hence is isomorphic to an $R$-module $M$ with 
$R\subseteq M\subseteq \tilde{R}$ 
(see e.\,g. \cite[Proof of Corollary C.3]{FGS1999}). 
Thus $\mathcal{F}_P$ is isomorphic to 
$I^{(P)}:=\mathrm{Hom}_R(M, R)\subseteq R$, 
which is an ideal of $R$ containing the conductor 
$\mathrm{Ann}(\tilde{R}/R)\cong 
\mathrm{Hom}_R(\tilde{R}, R)$. 
Patching $I^{(P)}$ for all $P\in C$, we obtain the desired $I_{Z_0}$. 

\medbreak
(2)
By the duality on fractional ideals, 
$\mathcal{O}_C\subseteq \mathcal{H}om_{\mathcal{O}_C}(I_{Z_0}, \mathcal{O}_C)\subseteq
\nu_*\mathcal{O}_{\tilde{C}}$ holds. 
It follows that $\mathrm{Hom}_{\mathcal{O}_C}(I_{Z_0}, \mathcal{O}_C)\cong\mathbb{C}$. 
By Grothendieck duality and Riemann-Roch formula, we have 
$h^1(\omega_C\otimes I_{Z_0})=1$ and 
$h^0(\omega_C\otimes I_{Z_0})=g-d$. 

Inductively, for $i=1, 2, \dots, g-d$, 
we take a general smooth point $R_i$ such that the map 
$H^0(\omega_C\otimes I_{Z_{i-1}})\to H^0((\omega_C\otimes I_{Z_{i-1}})|_{R_i})$ is nonzero, 
and set $Z_i=Z_{i-1}\cup \{R_i\}$. 
Then $h^0(\omega_C\otimes I_{Z_i})=g-d-i$ 
and $h^1(\omega_C\otimes I_{Z_i})=1$. 
For $Z:=Z_{g-d}$, 
we have 
$\mathrm{Hom}_{\mathcal{O}_C}(I_Z, \mathcal{O}_C)\cong\mathbb{C}$ and 
$\mathrm{Ext}^1_{\mathcal{O}_C}(I_Z, \mathcal{O}_C)=0$ 
by duality. 

\medbreak
(3)
We take $L$ to be $\mathcal{H}om_{\mathcal{O}_C}(I_Z, \mathcal{F})$ and 
extend it to a line bundle $\mathcal{L}$ on $\mathcal{C}$. 

Then we have $(A^g_{\mathcal{L}})^{-1}([\mathcal{F}])=\{[Z]\}$, 
and 
$\mathrm{Hom}_{\mathcal{O}_C}(\mathcal{F}, \mathcal{L}|_C)\cong\mathbb{C}$ and 
$\mathrm{Ext}^1_{\mathcal{O}_C}(\mathcal{F}, \mathcal{L}|_C)=0$ holds. 
By essentially the same arguments as in \cite[Theorem 5.18]{AK1980}, 
$A_{\mathcal{L}, Z}$ is an isomorphism. 
\end{proof}

We introduce some notations and terminologies 
concerning the families of germs of curves under consideration. 

\begin{definition}\label{def_w_contact}
Let $x, y$ denote the formal coordinates of $\hat{\mathbb{A}}^2$, 
and $D$ the divisor defined by $y$. 
Let $\Lambda$ be the formal spectrum of 
an object of $\hat{\mathcal{A}}_{\mathbb{C}}$. 
Hereafter, $\lambda$ symbolically represents a coordinate system on $\Lambda$, 
and $0$ the closed point. 
Thus a section of $\mathcal{O}_{\hat{\mathbb{A}}^2\times \Lambda}$ 
is written as $A_\lambda(x, y)$ and 
its restriction to the closed fiber as $A_0(x, y)$, and so on. 
\begin{enumerate}
\item
A \emph{family of $w$-contact equations} over $\Lambda$ 
is a section $\mathcal{E}$ of $\mathcal{O}_{\hat{\mathbb{A}}^2\times \Lambda}$ 
such that $\mathcal{E}|_{D\times \Lambda}=x^wu$ 
for a unit $u\in\mathcal{O}_{D\times \Lambda}^\times$. 
We write $\mathcal{E}$ as $E_{\lambda}(x, y)$ 
or $\{E_{\lambda}(x, y)\}_{\lambda\in\Lambda}$. 
Then we can write
\begin{equation}
E_{\lambda}(x, y) = yf_\lambda(x, y) + x^w g_\lambda(x) \label{eqn_w_contact}
\end{equation}
in a unique way. 
\item
A \emph{family of $w$-contact curves} over $\Lambda$ 
is a closed formal subscheme 
$\mathcal{C}=\{C_{\lambda}\}_{\lambda\in\Lambda}$
of $\hat{\mathbb{A}}^2\times \Lambda$ 
defined by a family of $w$-contact equations over $\Lambda$. 
\end{enumerate}
\end{definition}

Given a family $\{C_{\lambda}\}_{\lambda\in\Lambda}$ of $w$-contact curves 
defined by 
$E_{\lambda}(x, y) = yf_\lambda(x, y) + x^w g_\lambda(x)$, 
we can take another defining equation with $g_\lambda(x)=1$ 
by multiplying $g_\lambda(x)^{-1}$. 
By further multiplying a unit of the form $1+yh_\lambda(x, y)$, 
we may assume that $f_\lambda(x, y)$ is polynomial in $x$, of degree at most $w-1$. 
In other words, it is a monic polynomial of degree $w$ in $x$, 
and is the distinguished polynomial given by Weierstrass preparation theorem. 

\begin{definition}
A family of $w$-contact equations 
$E_{\lambda}(x, y) = yf_\lambda(x, y) + x^w g_\lambda(x)$ 
over $\Lambda$ is in \emph{normal form} 
if $g_\lambda(x)=1$, 
and is \emph{distinguished} 
if $g_\lambda(x)=1$ and 
$f_\lambda(x, y)$ is a polynomial of degree at most $w-1$ in $x$. 
\end{definition}

Let us look at the easiest case of our question. 

\begin{example}\label{ex_codim2}
Let $X$ be a rational surface, 
$D$ an anticanonical curve on $X$, $\beta$ a curve class on $X$ 
of arithmetic genus $g=p_a(\beta)$ 
and $\mathfrak{d}$ an effective divisor on $D_{\mathrm{sm}}$ 
with $\beta|_D\sim \mathfrak{d}$. 
Let $C_0\in \Lambda_{\beta, \mathfrak{d}}$ 
be nodal at $P\in\mathrm{Supp}\ \mathfrak{d}$, 
and $\mathcal{F}_0$ a torsion-free sheaf on $C_0$ invertible outside $P$ 
with $(\mathcal{F}_0)_P\cong (I_{P\subset C_0})_P$, 
such as $I_{P\subset C_0}$ itself. 
Then $[\mathcal{F}_0]\in\Moduli{X}{D}{\beta}{\mathfrak{d}}$ 
would be the easiest type of singularity. 

To see what this point looks like, 
assume that 
$\Lambda_{\beta, \mathfrak{d}}$ contains a curve $C_1$ nonsingular at $P$. 
Actually, this holds under our assumptions by Lemma \ref{lem_main_surjectivity} (2). 
Let $\mathcal{C}_{\beta, \mathfrak{d}}$ denote the total space of the family 
over $\Lambda_{\beta, \mathfrak{d}}$. 
By Proposition \ref{prop_twisted_abel_map}, 
a twisted Abel map gives a formal isomorphism from 
$([\mathcal{F}_0]\in\Moduli{X}{D}{\beta}{\mathfrak{d}})$ to 
$([Z]\in \mathrm{Hilb}^g(\mathcal{C}_{\beta, \mathfrak{d}}/\Lambda_{\beta, \mathfrak{d}}))$ 
where $Z=\{P, P_1\dots, P_{g-1}\}$ for distinct smooth points $P_1\dots, P_{g-1}\in C_0$. 
Further, 
the formal neighborhood 
$([Z]\in \mathrm{Hilb}^g(\mathcal{C}_{\beta, \mathfrak{d}}/\Lambda_{\beta, \mathfrak{d}}))$ 
is isomorphic to 
$(P\in \mathcal{C}_{\beta, \mathfrak{d}})\times \hat{\mathbb{A}}^{g-1}$, 
since $P_1\dots, P_{g-1}$ can move in the smooth parts of the fibers 
over $\Lambda_{\beta, \mathfrak{d}}$. 

We claim that $(P\in \mathcal{C}_{\beta, \mathfrak{d}})$ 
is formally isomorphic to $(\hbox{$A_{w-1}$-singularity})\times \hat{\mathbb{A}}^{g-1}$. 
In fact, $C_0$ can have any singularity here. 
Choose formal coordinates $x, y$ such that 
$D=(y=0)$ and $C_0=(yf_0(x, y)+x^w=0)$ with $f_0(0, 0)=0$. 
Then $C_1=(y+x^w v_1(x)=0)$ for some $v_1\in\mathbb{C}[[x]]^\times$, 
and $(P\in \mathcal{C}_{\beta, \mathfrak{d}})$ is defined by 
\[
(y f_0(x, y)+x^w) + s_1(y+x^w v_1(x))u_1(x, y)+\sum_{i=2}^g s_i (y f_i(x, y)+x^wg_i(x)),  
\]
where $s_1, \dots, s_g$ are coordinates on $\Lambda_{\beta, \mathfrak{d}}$, 
$u_1\in \mathbb{C}[[x, y]]^\times$, $f_i\in\mathbb{C}[[x, y]]$ and $g_i\in\mathbb{C}[[x]]$. 
We can rewrite this as 
\[
y\left(u_1(0, 0)s_1+ \varepsilon \right)
+ x^wU
\]
where $\varepsilon_1\in \langle x, y, s_2, \dots, s_g\rangle$ 
and $U\in \mathbb{C}[[x, y, s_1, \dots, s_g]]^\times$, 
so we have the claim. 
\end{example}

\section{A correspondence of Hilbert schemes}

In this section, we show that 
the completion of our relative Hilbert scheme at a point 
is smooth over 
a product of local Hilbert schemes of $A_n$-singularities 
under a certain condition. 
To give an idea, 
let us look at an example of codimension $4$ singularities. 

\subsection{A codimension $4$ example}

Let $C_0$ be a curve 
with a \emph{tacnode} at $P$. 
Choosing formal coordinates at $P$, 
the boundary $D$ is given by $y=0$ and the curve $C_0$ by 
\[
(y-u(x)x^2)(y-v(x)x^{w-2})=0, 
\]
where $u(x), v(x)\in\mathbb{C}[[x]]^\times$. 
Let $Z$ be the $0$-dimensional closed subscheme of $C_0$ 
defined by $\langle y, x^2\rangle$. 
Then, for a family $\mathcal{C}/\Lambda$ with central fiber $C_0$, 
the point $[Z]\in\mathrm{Hilb}^2(\mathcal{C}/\Lambda)$ would be the 
next type of singularity to consider. 

For a more specific example, 
let us consider the equation 
\[
C_{s, t}: F_{s, t}(x, y)=(y^2+x^4)+sx(y+x^3)+t(y+x^4)=0 
\]
of a family of $4$-contact curves 
over a $2$-dimensional space $\Lambda$ with coordinates $s, t$. 
This is an affine part of 
the linear system spanned by curves $y^2+x^4=0$ (tacnodal at $P$), 
$x(y+x^3)=0$ (nodal) and $y+x^4$ (nonsingular). 

Locally, we can embed the relative Hilbert scheme into $\mathbb{C}^6$ 
and write down a system of defining equations as follows. 
The Hilbert scheme $\mathrm{Hilb}^2(\mathbb{A}^2)$ 
contains an affine open subset $\mathbb{A}^4\subset\mathrm{Hilb}^2(\mathbb{A}^2)$ 
representing the family of closed subschemes 
\[
Z_{k, l, m, n}: y-kx-l=0,\quad x^2-mx-n=0 
\]
of length $2$. 
The relative Hilbert scheme is locally given by
\[
\mathcal{H}=\{(s, t, k, l, m, n)\mid Z_{k, l, m, n}\subset C_{s, t}\}\subset\mathbb{C}^6. 
\]
Substituting $y=kx+l$ into the equation for $C_{s, t}$ 
and reducing it modulo $x^2-mx-n$, 
we obtain a linear polynomial in $x$. 
The coefficients of $1$ and $x$ provide a system of equations for $\mathcal{H}$: 
\begin{align*}
& n(sm^2+sk+k^2+m^2)+tm^2n+tl+l^2+n^2(1+s+t)=0, \\
& m(sm^2+sk+k^2+m^2)+tm^3+tk+sl+2kl+2mn(1+s+t)=0. 
\end{align*}
One can check that 
the singular locus $\mathrm{Sing}(\mathcal{H})$ is given by 
\[
t=l=n=0,\quad sm^2+sk+k^2+m^2=0, 
\]
and that at a general point of $\mathrm{Sing}(\mathcal{H})$, 
$\mathcal{H}$ is formally isomorphic to $(\hbox{$A_3$-singularity})\times \mathbb{C}^2$. 
The singular locus $\mathrm{Sing}(\mathcal{H})$ itself is irreducible 
with an $A_1$-singularity at the origin. 

Let us show that $(0\in\mathcal{H})$ is formally isomorphic to the point 
$([\tilde{Z}]\in\mathrm{Hilb}^2(S))$, 
where $S=(yz+x^4=0)$ is an $A_3$-singularity 
and $\tilde{Z}\subset S$ is defined by $\langle z, y, x^2\rangle$, 
using the following ``trick.'' 
The condition $Z_{k, l, m, n}\subset C_{s, t}$ is expressed as 
\[
(y^2+x^4)+sx(y+x^3)+t(y+x^4)\in \langle y-kx-l, x^2-mx-n\rangle. 
\]
The left hand side is equal to 
$y(y+sx+t)+(1+s+t)x^4$, and since $y\equiv kx+l$ modulo the right hand side, 
this is equivalent to 
\[
y\{(s+k)x+(t+l)\}+(1+s+t)x^4\in \langle y-kx-l, x^2-mx-n\rangle, 
\]
and in a neighborhood of $(s=t=0)$, also to 
\[
y\{(1+s+t)^{-1}(s+k)x+(1+s+t)^{-1}(t+l)\}+x^4\in \langle y-kx-l, x^2-mx-n\rangle. 
\]
Now we change the coordinates of the parameter space 
from $(k, l, m, n, s, t)$ to $(k, l, m, n, s', t')$ where 
\[
s'=(1+s+t)^{-1}(s+k),\ t'=(1+s+t)^{-1}(t+l), 
\]
and introducing a new variable $z$, the condition is equivalent to 
\[
yz+x^4\in \langle z-s'x-t', y-kx-l, x^2-mx-n\rangle. 
\]
The last condition gives equations for a neighborhood of $[\tilde{Z}]$ in $\mathrm{Hilb}^2(S)$, 
and we have shown the claim. 
Roughly, what we have done can be seen as lifting the ideals to the $z$-direction 
using the parameters for the curves.

\subsection{A nondegeneracy condition}

To generalize the above example, 
we first formulate a certain nondegeneracy condition 
for families of curves. 
Let $\mathcal{O}=\mathbb{C}[[x, y]]$ in what follows.

\begin{definition}\label{def_Phi_nondegeneracy}
Let $\Lambda$ be the formal spectrum of an object of $\hat{\mathcal{A}}_{\mathbb{C}}$ 
which is regular, 
$0$ its closed point, 
$\mathcal{E}=\{E_{\lambda}\}_{\lambda\in \Lambda}$ 
an element of $\mathcal{O}_{\hat{\mathbb{A}}^2\times\Lambda}$ 
and $I\subseteq \mathcal{O}$ an ideal of finite colength containing $E_0$. 
\begin{enumerate}
\item
We define maps 
\[
\tilde{\Delta}_{\mathcal{E}}: 
T_0\Lambda \to \mathcal{O};  
\ \bv\mapsto 
\bv (E_{\lambda}(x, y))
\]
and 
\[
\Delta_{\mathcal{E}, I}: 
T_0\Lambda \to \mathcal{O}/I; 
\ \bv\mapsto 
\tilde{\Delta}_{\mathcal{E}}(\bv)
+ I. 
\]
Here $\bv$ is regarded as a partial derivation acting on the $\lambda$-coordinates. 
\item
Assume that $\mathcal{E}=\{E_{\lambda}\}_{\lambda\in \Lambda}$ 
is a family of $w$-contact equations (Definition \ref{def_w_contact}) given as 
\[
E_{\lambda}(x, y)=y f_\lambda(x, y) + x^w g_\lambda(x). 
\]
We define maps 
\[
\tilde{\Phi}_{\mathcal{E}}: 
T_0\Lambda \to \mathcal{O};  
\ \bv\mapsto 
\bv \left( 
\frac{f_\lambda(x, y)}{g_\lambda(x)}
\right)
= 
\bv \left( 
\frac{(E_{\lambda}(x, y)-E_{\lambda}(x, 0))/y}{E_{\lambda}(x, 0)/x^w} 
\right)
\]
and 
\[
\Phi_{\mathcal{E}, I}: 
T_0\Lambda \to \mathcal{O}/I; 
\ \bv\mapsto 
\tilde{\Phi}_{\mathcal{E}}(\bv)
+ I. 
\]
Then we consider the following condition: 
\begin{center}
($\ast$) The map $\Phi_{\mathcal{E}, I}$ is surjective. 
\end{center}
\end{enumerate}
We also write $\Delta_{\mathcal{E}, I_Z}$ and $\Phi_{\mathcal{E}, I_Z}$ 
as $\Delta_{\mathcal{E}, Z}$ and $\Phi_{\mathcal{E}, Z}$. 
\end{definition}

We see that the condition ($\ast$) is invariant under certain coordinate changes, 
and that it is determined by the family of curves defined by $\mathcal{E}$. 

\begin{definition}
We say that an automorphism $\varphi$ of $\mathcal{O}$ 
is \emph{strata-preserving} if it maps 
$\langle y\rangle$ to $\langle y\rangle$ 
and $\langle x, y\rangle$ to $\langle x, y\rangle$. 
\end{definition}

\begin{proposition}\label{prop_invariance}
\begin{enumerate}
\item
The condition ($\ast$) is invariant under a multiplication 
by a unit $u_\lambda(x, y)$, 
\[
E_{\lambda}(x, y)\leadsto u_\lambda(x, y)E_{\lambda}(x, y). 
\]
In other words, the condition ($\ast$) depends only on the family 
of $w$-contact curves $\{C_\lambda\}_{\lambda\in \Lambda}$ 
defined by $\{E_{\lambda}\}_{\lambda\in \Lambda}$ 
as a closed formal subscheme of $\hat{\mathbb{A}}^2\times\Lambda$. 
\item
The condition ($\ast$) is invariant under a 
strata-preserving change of coordinates $(x, y)$: 
if $\varphi: \mathcal{O}\to\mathcal{O}$ is a strata-preserving automorphism, 
then $\Phi_{\mathcal{E},  I}$ is surjective if and only if 
$\Phi_{\varphi(\mathcal{E}), \varphi(I)}$ is surjective, 
where $\varphi(\mathcal{E}):=\{\varphi(E_{\lambda}(x, y))\}_{\lambda\in\Lambda}$. 
Here $\varphi$ is identified with $\varphi\hat{\otimes} id_{\mathcal{O}_\Lambda}$. 
\end{enumerate}
\end{proposition}
\begin{proof}
(1)
We write $u_\lambda(x, y)$ as 
$u_\lambda(x, y)=\bar{u}_\lambda(x)\{1+yh_\lambda(x, y)\}$. 
It is obvious that the map $\Phi_{\mathcal{E}, I}$ 
does not change under $E_{\lambda}(x, y)\leadsto\bar{u}_\lambda(x)E_{\lambda}(x, y)$, 
so we may assume that 
$u_\lambda(x, y)=1+yh_\lambda(x, y)$. 

If $E_{\lambda}(x, y)=y f_\lambda(x, y) + x^w g_\lambda(x)$, 
then 
\begin{eqnarray*}
u_\lambda(x, y)E_{\lambda}(x, y)
& = & 
\{1+yh_\lambda(x, y)\} \{y f_\lambda(x, y) + x^w g_\lambda(x)\} \\
& = & 
y \{(1 + yh_\lambda(x, y))f_\lambda(x, y)+x^w g_\lambda(x)h_\lambda(x, y)\} 
+ x^w g_\lambda(x). 
\end{eqnarray*}
Thus $\tilde{\Phi}_{\{u_\lambda E_{\lambda}\}}$ sends $\bv$ to 
\begin{eqnarray*}
& & 
\bv\left(
\frac{(1 + yh_\lambda(x, y))f_\lambda(x, y)+x^w g_\lambda(x)h_\lambda(x, y)}{g_\lambda(x)}
\right) \\
& = &
\bv\left(
(1 + yh_\lambda(x, y))\frac{f_\lambda(x, y)}{g_\lambda(x)} + x^w h_\lambda(x, y)
\right) \\
& = &
\bv(1 + yh_\lambda(x, y))\frac{f_0(x, y)}{g_0(x)}
+
(1 + yh_0(x, y))
\bv\left(
\frac{f_\lambda(x, y)}{g_\lambda(x)}
\right)
+
x^w 
\bv\left(
h_\lambda(x, y)
\right) \\
& = &
\bv(h_\lambda(x, y))\left(y\frac{f_0(x, y)}{g_0(x)} + x^w\right)
+
(1 + yh_0(x, y))
\bv\left(
\frac{f_\lambda(x, y)}{g_\lambda(x)}
\right) \\
& \equiv &
(1 + yh_0(x, y))
\bv\left(
\frac{f_\lambda(x, y)}{g_\lambda(x)}
\right) \mod I, 
\end{eqnarray*}
by 
$yf_0(x, y) + x^w g_0(x)\in I$. 
Since $1 + yh_0(x, y)$ is a unit, 
$\Phi_{\{u_\lambda E_{\lambda}\}, I}$ is surjective if and only if 
$\Phi_{\{E_{\lambda}\}, I}$ is. 

\medskip
(2)
Let $X:=\varphi(x), Y:=\varphi(y)$ and write 
$X=x\cdot u(x)+y\cdot A(x, y)$ and $Y=y\cdot v(x, y)$, 
where $u(x)$ and $v(x, y)$ are units. 
Since multiplying a unit does not affect the surjectivity on the both sides by (1), 
we may assume that $g_\lambda(x)=1$. 
Then 
\begin{eqnarray*}
& & \varphi(yf_\lambda(x, y)+x^w) \\ 
& = & Y f_\lambda(X, Y)+X^w \\
& = & yv(x, y)f_\lambda(X, Y)+(xu(x)+yA(x, y))^w \\
& = & 
y\left\{v(x, y)f_\lambda(X, Y)  + 
\sum_{i=1}^w \begin{pmatrix} w \\ i \end{pmatrix} x^{w-i}u(x)^{w-i}y^{i-1}A(x, y)^i
\right\}
+ x^w u(x)^w. 
\end{eqnarray*}
Thus $\tilde{\Phi}_{\varphi(\mathcal{E})}$ sends $\bv$ to 
\begin{eqnarray*}
& & 
\bv\left(
u(x)^{-w}v(x, y)f_\lambda(X, Y) + 
\sum_{i=1}^w \begin{pmatrix} w \\ i \end{pmatrix} 
x^{w-i}u(x)^{-i}y^{i-1}A(x, y)^i
\right) \\
& = &
u(x)^{-w}v(x, y)\bv(f_\lambda(X, Y)) \\
& = &
u(x)^{-w}v(x, y)\varphi(\bv(f_\lambda(x, y))), 
\end{eqnarray*}
and $\Phi_{\varphi(\mathcal{E}), \varphi(I)}$ is surjective 
if and only if $\Phi_{\mathcal{E}, I}$ is. 
\end{proof}

As we will see in Theorem \ref{thm_corr}, 
this condition ($\ast$) is a sufficient condition for 
our relative Hilbert scheme to have symplectic singularities. 
We can somewhat relax the condition 
by using families of automorphisms which depend on the parameter $\lambda$, 
although we do not need this for the proof of Theorem \ref{thm_main}. 
For simplicity, we only consider equations in the normal form. 

\begin{proposition}\label{prop_coord_change}
Let $\Lambda$ be the formal spectrum of an object of $\hat{\mathcal{A}}_{\mathbb{C}}$ 
which is regular. 
Let $\mathcal{E}=\{E_{\lambda}\}_{\lambda\in \Lambda}$ 
be a family of $w$-contact equations over $\Lambda$ 
in the normal form 
\[
E_{\lambda}(x, y)=y f_\lambda(x, y) + x^w. 
\]
Then the following conditions are equivalent. 
\begin{enumerate}
\item[(a)]
There exists a family 
$\tilde{\varphi}: \mathcal{O}_{\hat{\mathbb{A}}^2\times \Lambda}
\to \mathcal{O}_{\hat{\mathbb{A}}^2\times \Lambda}$ of strata-preserving automorphisms 
of $\mathcal{O}_{\hat{\mathbb{A}}^2}$ over $\Lambda$ 
such that $\tilde{\varphi}(\mathcal{E})$ and $\tilde{\varphi}_0(I)$ satisfy ($\ast$). 
\item[(b)]
The map 
\[
T_0\Lambda \to 
\mathcal{O}\left/\left(
I 
+ 
\left\langle 
 x\frac{\partial f_0(x, y)}{\partial x} - wf_0(x, y), 
\frac{\partial E_0(x, y)}{\partial x}, 
\frac{\partial E_0(x, y)}{\partial y}
\right\rangle
\right)
\right.
\]
induced by $\Phi_{\mathcal{E}, I}$ is surjective. 
\end{enumerate}
\end{proposition}

\begin{proof}
By Proposition \ref{prop_invariance}, 
we have only to consider automorphisms which restrict to the identity at $\lambda=0$. 
In other words, for $X:=\tilde{\varphi}(x), Y:=\tilde{\varphi}(y)$, 
we assume 
$X=x\cdot u_\lambda(x)+y\cdot A_\lambda(x, y)$ and $Y=y\cdot v_\lambda(x, y)$, 
where $u_0(x)=1$, $A_0(x, y)=0$ and $v_0(x, y)=1$. 

As in the proof of Proposition \ref{prop_invariance}, 
$\tilde{\varphi}(yf_\lambda(x, y)+x^w)$ is 
\[
y\left\{f_\lambda(X, Y)v_\lambda(x, y)  + 
\sum_{i=1}^w \begin{pmatrix} w \\ i \end{pmatrix} x^{w-i}u_\lambda(x)^{w-i}y^{i-1}A_\lambda(x, y)^i
\right\}
+ x^w u_\lambda(x)^w, 
\]
and $\tilde{\Phi}_{\tilde{\varphi}(\mathcal{E})}$ sends $\bv$ to 
\begin{eqnarray*}
& & 
\bv\left(
f_\lambda(X, Y)u_\lambda(x)^{-w}v_\lambda(x, y) + 
\sum_{i=1}^w \begin{pmatrix} w \\ i \end{pmatrix} 
x^{w-i}u_\lambda(x)^{-i}y^{i-1}A_\lambda(x, y)^i
\right) \\
& = &
\bv(f_\lambda(X, Y))\cdot u_0(x)^{-w}v_0(x, y) \\
& & 
+ (-w)f_0(x, y)\cdot\bv(u_\lambda(x))\cdot(u_0(x))^{-w-1}v_0(x, y) \\
& & 
+ f_0(x, y)u_0(x)^{-w}\cdot\bv(v_\lambda(x, y))  \\
& &
+ \sum_{i=1}^w \begin{pmatrix} w \\ i \end{pmatrix} 
x^{w-i}\bv(u_\lambda(x))\cdot(-i)(u_\lambda(x))^{-i-1}y^{i-1}A_0(x, y)^i \\
& & 
+ \sum_{i=1}^w \begin{pmatrix} w \\ i \end{pmatrix} 
x^{w-i}u_0(x)^{-i}y^{i-1}\bv(A_\lambda(x, y))\cdot i A_0(x, y)^{i-1} \\
& = &
\bv(f_\lambda(x, y)) 
+ \bv(X)\frac{\partial f_0(x, y)}{\partial x}
+ \bv(Y)\frac{\partial f_0(x, y)}{\partial y}
\\
& & 
+ (-w)f_0(x, y)\bv(u_\lambda(x)) 
+ f_0(x, y)\bv(v_\lambda(x, y)) 
+ w x^{w-1}\bv(A_\lambda(x, y)) \\
& = &
\bv(f_\lambda(x, y)) \\
& & 
+ \left(x\frac{\partial f_0(x, y)}{\partial x} -wf_0(x, y) \right) \bv(u_\lambda(x))
+ \left(y\frac{\partial f_0(x, y)}{\partial x} + w x^{w-1} \right) \bv(A_\lambda(x, y)) \\ 
& & 
+ \left(y \frac{\partial f_0(x, y)}{\partial y} + f_0(x, y)\right)\bv(v_\lambda(x, y)) \\ 
& = &
\tilde{\Phi}_{\mathcal{E}}(\bv) 
+ \left(x\frac{\partial f_0(x, y)}{\partial x} -wf_0(x, y) \right) \bv(u_\lambda(x)) \\
& & 
+ \frac{\partial E_0(x, y)}{\partial x}\bv(A_\lambda(x, y)) 
+ \frac{\partial E_0(x, y)}{\partial y} \bv(v_\lambda(x, y)). 
\end{eqnarray*}
The sum of the last three terms exhausts 
\[
J:= 
\left(x\frac{\partial f_0(x, y)}{\partial x} -wf_0(x, y) \right)\mathbb{C}[[x]]
+ 
\frac{\partial E_0(x, y)}{\partial x}\mathcal{O} 
+
\frac{\partial E_0(x, y)}{\partial y}\mathcal{O} 
\]
when $u_\lambda(x)$, $A_\lambda(x, y)$ and $v_\lambda(x, y)$ vary. 
We see that $E_0\mathcal{O} + J$ 
is an ideal in $\mathcal{O}$ from 
\begin{equation}
y\left( x\frac{\partial f_0(x, y)}{\partial x} - wf_0(x, y) \right)
= 
x\frac{\partial E_0(x, y)}{\partial x} - wE_0(x, y). 
\label{eqn_relation}
\end{equation}
Thus $I+J$ is an ideal, equal to 
\[
I 
+ \left\langle 
x\frac{\partial f_0(x, y)}{\partial x} -wf_0(x, y), 
\frac{\partial E_0(x, y)}{\partial x}, 
\frac{\partial E_0(x, y)}{\partial y}
\right\rangle. 
\]

If $\tilde{\varphi}(\mathcal{E})$ and $I$ satisfy ($\ast$), 
i.\,e. $\Phi_{\tilde{\varphi}(\mathcal{E}), I}$ is surjective, 
it follows that 
$\mathrm{Im}\,\tilde{\Phi}_{\mathcal{E}} + I+J = \mathcal{O}$. 

Conversely, assume that $\mathrm{Im}\,\tilde{\Phi}_{\mathcal{E}} + I+J = \mathcal{O}$. 
Let $d$ be the length of $\mathcal{O}/I$. 
Then 
there exist a nonnegative integer $k$, 
a formal coordinate system $\lambda_1, \dots, \lambda_r$  on $\Lambda$ 
and $u^{(j)}(x)\in\mathbb{C}[[x]]$, 
$A^{(j)}(x, y)\in\mathcal{O}$ and $v^{(j)}(x, y)\in\mathcal{O}$ 
for $j=1, \dots, k$ such that 
the images of 
\[
\left(x\frac{\partial f_0(x, y)}{\partial x} - wf_0(x, y) \right)u^{(j)}(x)
+ 
\frac{\partial E_0(x, y)}{\partial x}A^{(j)}(x, y) 
+ 
\frac{\partial E_0(x, y)}{\partial y}v^{(j)}(x, y) 
\]
for $j=1, \dots, k$ 
and 
$\Phi_{\mathcal{E}, I}(\partial/\partial \lambda_j)$ for $j=k+1, \dots, d$ 
span $\mathcal{O}/I$. 
Then, for a general $c\in\mathbb{C}$, 
the map 
$\Phi_{\tilde{\varphi}(\mathcal{E}), I}$ is surjective 
for the automorphism $\tilde{\varphi}$ given by 
$u_\lambda(x)=1+c\sum_{j=1}^k\lambda_j u^{(j)}(x)$, 
$A_\lambda(x, y)=1+c\sum_{j=1}^k\lambda_j A^{(j)}(x, y)$ 
and $v_\lambda(x, y)=1+c\sum_{j=1}^k\lambda_j v^{(j)}(x, y)$. 
\end{proof}

\begin{remark}\label{rmk_rephrase}
Let 
$\mathcal{T}:=\langle x, y\rangle \partial_x + \langle y\rangle \partial_y$ 
be the module of tangent fields fixing $(y=0)$ and $(x=y=0)$. 
By \eqref{eqn_relation}, the following gives a well-defined homomorphism 
$\Psi_{E_0, I}: \mathcal{T}\to\mathcal{O}/I$ of $\mathcal{O}$-modules: for $\xi_1, \xi_2, \eta\in\mathcal{O}$, 
\begin{eqnarray*}
D=(x\xi_1+y\xi_2)\partial_x + y\eta \partial_y
& \mapsto & 
\xi_1\left( x \frac{\partial f_0}{\partial x} - wf_0 \right)  + 
           \xi_2 \frac{\partial E_0}{\partial x}  + 
           \eta \frac{\partial E_0}{\partial y}  + I. 
\end{eqnarray*}
The condition of Proposition \ref{prop_coord_change} 
can be rephrased 
as the surjectivity of the map 
\[
\Phi_{\mathcal{E}, I} + \Psi_{E_0, I}: 
T_{0}\Lambda \times \mathcal{T} \to \mathcal{O}/I; 
\ (\bv, D) \mapsto 
\Phi_{\mathcal{E}, I}(\bv) + \Psi_{E_0, I}(D). 
\]
We may also take 
$T_0\Lambda \times \mathcal{T}/I\mathcal{T}$
as the domain of the map. 

To some extent, this condition is analogous to 
the criterion for the nonsingularity 
of $\mathrm{Hilb}^d(\mathcal{C}/\Lambda)$ 
at $[I]\in\mathrm{Hilb}^d(C_0)$ 
given at the beginning of \cite[\S4]{Shende2012}: 
under the usual identification 
$T_{[I]}\mathrm{Hilb}^d(\hat{\mathbb{A}}^2)\cong 
\mathrm{Hom}_{\mathcal{O}}(I, \mathcal{O}/I)$, 
$[I]\in\mathrm{Hilb}^d(\mathcal{C}/\Lambda)$ is nonsingular 
if and only if the map 
\[
T_0\Lambda\times T_{[I]}\mathrm{Hilb}^d(\hat{\mathbb{A}}^2)\to\mathcal{O}/I; 
\ (\bv, \eta)\mapsto\eta(E_0)-(\bv(E_\lambda)+I)
\]
is surjective, 
where $\{E_\lambda\}$ is a defining equation for $\mathcal{C}$. 
\end{remark}

\begin{remark}
Let $R=\mathcal{O}/\langle E_0\rangle$, $\tilde{R}$ its normalization, 
$C\subseteq R$ the conductor 
and $J\subseteq \mathcal{O}$ the Jacobian ideal. 
By \cite[p. 261]{Piene1978}, 
$J\cdot \tilde{R}=(C\cdot\tilde{R})\cdot \mathrm{Fitt}^0(\Omega_{\tilde{R}/R})
\subset C\cdot \tilde{R}=C$, 
and hence 
$J\cdot R\subseteq C$ holds. 
Thus, if $I\cdot R$ contains the conductor ideal, 
then 
\[
I+ 
\left\langle 
 x\frac{\partial f_0(x, y)}{\partial x} - wf_0(x, y), 
\frac{\partial E_0(x, y)}{\partial x}, 
\frac{\partial E_0(x, y)}{\partial y}
\right\rangle
= I + \left\langle 
 x\frac{\partial f_0(x, y)}{\partial x} - wf_0(x, y) 
\right\rangle
\]
holds. 
In many examples 
we also have $x{\partial f_0(x, y)}/{\partial x} - wf_0(x, y)\in C$, 
but the author does not know if this is always the case. 
\end{remark}

\subsection{Correspondence of Hilbert schemes}

Let us start with the morphisms between formal Hilbert schemes 
associated to a smooth morphism. 

\begin{definition}
For a (formal) scheme $X$ over $\mathbb{C}$ 
and a $0$-dimensional subscheme $Z$ of $X$, 
let $u^X_Z: \mathcal{U}^X_Z\to\mathcal{H}^X_Z$ 
be the universal family over the formal Hilbert scheme, 
and $o^X_Z: \mathbb{O}^X_Z\to \mathcal{H}^X_Z$ 
the $\mathbb{A}^d$-bundle associated to $(u^X_Z)_*\mathcal{O}_{\mathcal{U}^X_Z}$ 
in the covariant way, 
where $d$ is the length of $Z$. 
\end{definition}

\begin{definition}
Let $\pi: \tilde{X}\to X$ be a morphism. 
If $Z\subseteq X$ is a closed subscheme, 
a \emph{lift} of $Z$ is a closed subscheme $\tilde{Z}\subset \tilde{X}$
such that $\pi$ induces an isomorphism $\tilde{Z}\overset{\sim}\to Z$. 
We use the same term for the corresponding ideals. 
We can also consider lifts of flat families of closed subschemes of $X$. 
\end{definition}

\begin{proposition}\label{prop_hilb_and_smooth_morphism}
Let $\pi: \tilde{X}\to X$ be a morphism 
between formal schemes which are formally of finite type over $\mathbb{C}$, 
$Z$ a closed subscheme of $X$ of finite length. 
\begin{enumerate}
\item
Let $\tilde{Z}\subseteq \tilde{X}$ be a lift of $Z$. 
Then there is a morphism 
$\pi_{\tilde{Z}}: \mathcal{H}^{\tilde{X}}_{\tilde{Z}} \to \mathcal{H}^X_Z$ 
characterized by the following property: 
if $T$ is the spectrum of an object of $\mathcal{A}_\mathbb{C}$, 
$\tilde{\mathcal{Z}}\subseteq \tilde{X}\times T$ is a family of subschemes of $\tilde{X}$ 
corresponding to $f: T\to \mathcal{H}^{\tilde{X}}_{\tilde{Z}}$
and $\mathcal{Z}\subseteq X\times T$ is the family of subschemes of $X$ 
induced by the morphism $\pi_{\tilde{Z}}\circ f$, 
then $\pi\times id_T: \tilde{X}\times T\to X\times T$ 
induces an isomorphism $\tilde{\mathcal{Z}}\to\mathcal{Z}$. 
\item
Let $\tilde{X}$ be $X\times \mathbb{A}^1$ and 
$\pi: \tilde{X}\to X$ the projection. 
Then lifts of $Z$ are in one-to-one correspondence 
with the $\mathbb{C}$-valued points of $\mathbb{O}^X_Z$. 
If $\tilde{Z}$ is a lift of $Z$ and $P_{\tilde{Z}}\in \mathbb{O}^X_Z$ 
is the corresponding point, 
$\mathcal{H}^{\tilde{X}}_{\tilde{Z}}$ is isomorphic 
to $(\mathbb{O}^X_Z)^\wedge_{P_{\tilde{Z}}}$, 
the completion of $\mathbb{O}^X_Z$ at $P_{\tilde{Z}}$, 
in a natural way. 
\end{enumerate}
\end{proposition}

\begin{proof}
(1) 
Let $\tilde{\mathcal{Z}}\subseteq \tilde{X}\times T$ be 
as in the statement. 
Since its central fiber $\tilde{Z}$ is a lift of $Z$, 
the homomorphism $\mathcal{O}_{X}\to \mathcal{O}_{\tilde{Z}}$ 
is surjective. 
By Nakayama's lemma, 
$\mathcal{O}_{X\times T}\to \mathcal{O}_{\tilde{\mathcal{Z}}}$ 
is surjective, 
i.\,e. $\tilde{\mathcal{Z}}\to X\times T$ is a closed immersion, 
and $\tilde{\mathcal{Z}}$ is a lift of the family $\mathcal{Z}\subseteq X\times T$ 
of closed subschemes of $X$ defined by the kernel of 
$\mathcal{O}_{X\times T}\to \mathcal{O}_{\tilde{\mathcal{Z}}}$. 
This assignment $\tilde{\mathcal{Z}}\mapsto\mathcal{Z}$ is functorial, 
so we have a morphism 
$\pi_{\tilde{Z}}: \mathcal{H}^{\tilde{X}}_{\tilde{Z}} \to \mathcal{H}^X_Z$ 
as stated. 

\medskip
(2)
Let $z$ be the coordinate function on $\mathbb{A}^1$. 
If $\tilde{Z}$ is a lift of $Z$, 
then $z|_{\tilde{Z}}$ can be regarded as a section of 
$\mathcal{O}_Z$, hence as a $\mathbb{C}$-valued point $P_{\tilde{Z}}$ 
of $\mathbb{O}^X_Z$. 

If $\tilde{\mathcal{Z}}\subseteq \tilde{X}\times T$ is a family 
in $\mathcal{H}^{\tilde{X}}_{\tilde{Z}}$, 
then it lifts a family $\mathcal{Z}\subseteq X\times T$ by (1), 
and $\bar{s}:=z|_{\tilde{\mathcal{Z}}}$ 
can be regarded as a section of $\mathcal{O}_{\mathcal{Z}}$. 
We take $s\in\mathcal{O}_{X\times T}$ 
such that $s|_{\mathcal{Z}}=\bar{s}$, 
and then the ideal of $\tilde{\mathcal{Z}}$ is generated by 
$I_{\mathcal{Z}\subseteq X\times T}$ and $z-s$. 
The restriction of $\bar{s}$ to the central fiber $Z$ of $\mathcal{Z}$ 
corresponds to $P_{\tilde{Z}}$, 
and so $\bar{s}$ gives a map $T\to (\mathbb{O}^X_Z)^\wedge_{P_{\tilde{Z}}}$. 
Conversely, if a morphism $T\to (\mathbb{O}^X_Z)^\wedge_{P_{\tilde{Z}}}$ is given, 
then it defines a family 
$u: \mathcal{Z}\hookrightarrow X\times T\to T$ of closed subschemes of $X$ 
and a section of $u_*\mathcal{O}_{\mathcal{Z}}$ 
whose restriction to the closed point of $T$ corresponds to $\tilde{Z}$. 
The latter gives a morphism $\mathcal{Z}\to \mathbb{A}^1$, 
giving rise to a closed embedding $\mathcal{Z}\hookrightarrow \tilde{X}\times T$ 
which is a lift of $\mathcal{Z}\hookrightarrow X\times T$ 
and extends $\tilde{Z}\hookrightarrow \tilde{X}$. 

These correspondences are inverse to each other, 
and so we have an isomorphism 
$\mathcal{H}^{\tilde{X}}_{\tilde{Z}} \cong (\mathbb{O}^X_Z)^\wedge_{P_{\tilde{Z}}}$. 
\end{proof}

Let $\mathcal{E}=\{E_{\lambda}\}_{\lambda\in\Lambda}$ 
be a family of elements of $\mathcal{O}=\mathbb{C}[[x, y]]$ 
parametrized by $\Lambda$, 
and let $\mathcal{C}\to\Lambda$ be the associated family of curves. 
We have natural morphisms 
\[
\mathrm{Hilb}^d(\mathcal{C}/\Lambda)
\hookrightarrow \mathrm{Hilb}^d(\hat{\mathbb{A}}^2_{\Lambda}/\Lambda)
\cong \mathrm{Hilb}^d(\hat{\mathbb{A}}^2)\times\Lambda
\to \mathrm{Hilb}^d(\hat{\mathbb{A}}^2)
\]
and 
$\mathrm{Hilb}^d(\hat{\mathbb{A}}^3_{\Lambda}/\Lambda)
\to \mathrm{Hilb}^d(\hat{\mathbb{A}}^3)$. 

\begin{definition}
Let $\mathcal{E}=\{E_{\lambda}\}_{\lambda\in\Lambda}$ 
be a family of $w$-contact equations over $\Lambda$, 
with $E_{\lambda}(x, y)=yf_\lambda(x, y)+w^w g_\lambda(x)$. 

We define $\mathcal{S}_{\mathcal{E}}$  
to be the closed formal subscheme of $\hat{\mathbb{A}}^3_{\Lambda}$ 
given by 
$\langle z-g_\lambda(x)^{-1}f_\lambda(x, y)\rangle
=\langle f_\lambda(x, y) - zg_\lambda(x)\rangle$. 
Here, we take the completion of $\mathbb{A}^3$ 
at $x=y=0, z=g_0(0)^{-1}f_0(0, 0)$. 
The isomorphism 
$\mathcal{S}_{\mathcal{E}}\to \hat{\mathbb{A}}^2_{\Lambda}$ defined by the projection 
induces an isomorphism 
\[
\mathrm{Hilb}^d(\mathcal{S}_{\mathcal{E}}/\Lambda)
\to \mathrm{Hilb}^d(\hat{\mathbb{A}}^2_{\Lambda}/\Lambda). 
\]
By composing the inverse of this isomorphism 
with the inclusion into $\mathrm{Hilb}^d(\hat{\mathbb{A}}^3_{\Lambda}/\Lambda)$ 
and the projection to $\mathrm{Hilb}^d(\hat{\mathbb{A}}^3)$, 
we obtain a morphism 
\[
L_{\mathcal{E}}^d: \mathrm{Hilb}^d(\hat{\mathbb{A}}^2_{\Lambda}/\Lambda)
\to \mathrm{Hilb}^d(\hat{\mathbb{A}}^3). 
\]
If $I\subset\mathcal{O}$ is an ideal of colength $d$ 
and $[\tilde{I}]=L_{\mathcal{E}}^d([I])$, 
we write 
\[
L_{\mathcal{E}, I}: 
\FormalHilb{\hat{\mathbb{A}}^2_{\Lambda}/\Lambda}{I}
\to \FormalHilb{\hat{\mathbb{A}}^3}{\tilde{I}} 
\]
for the morphism between the completions. 
This is a morphism over $\FormalHilb{\hat{\mathbb{A}}^2}{I}$. 
\end{definition}

\begin{lemma}
Let $\mathcal{C}$ be the family of curves defined by $\mathcal{E}$. 
The restriction of $L_{\mathcal{E}, I}$ 
to $\FormalHilb{\mathcal{C}/\Lambda}{I}$ 
depends only on $\mathcal{C}$ as a formal subscheme 
of $\hat{\mathbb{A}}^2_{\Lambda}$. 
\end{lemma}
\begin{proof}
Let $R=\mathcal{O}_{\Lambda}$, 
$\mathcal{O}_R=R[[x, y]]$, 
$\tilde{\mathcal{O}}=\mathbb{C}[[x, y, z]]$ 
and $\tilde{\mathcal{O}}_R=R[[x, y, z]]$. 
Take an object $A$ of $\mathcal{A}_R$ 
and a family $\mathcal{I}\subset \mathcal{O}_R\otimes_R A$($\cong \mathcal{O}\otimes A$) 
over $A$ 
belonging to 
$\FormalHilb{\mathcal{C}/\Lambda}{I}$. 
Then $L_{\mathcal{E}, I}$ maps the $A$-valued point $[\mathcal{I}]$ to 
the point given by the ideal 
\[
\tilde{\mathcal{I}}:=\mathcal{I}\cdot\tilde{\mathcal{O}}_R\otimes_R A+ 
(f_\lambda(x, y) - zg_\lambda(x))\cdot\tilde{\mathcal{O}}_R\otimes_R A
\subset \tilde{\mathcal{O}}_R\otimes_R A. 
\]

To prove the statement, it suffices to show that 
replacing $\mathcal{E}=\{E_\lambda(x, y)\}$ by 
$\mathcal{E}'=\{u_\lambda(x, y)E_\lambda(x, y)\}$ 
leads to the same $\tilde{\mathcal{I}}$, where $u_\lambda(x, y)=1+yh_\lambda(x, y)$ 
as in the proof of Proposition \ref{prop_invariance}. 
We have 
\begin{eqnarray*}
u_\lambda(x, y)E_{\lambda}(x, y)
& = & 
y \{f_\lambda(x, y)+h_\lambda(x, y)E_\lambda(x, y)\} 
+ x^w g_\lambda(x), 
\end{eqnarray*}
so $f_\lambda(x, y)- z g_\lambda(x)$ is replaced by 
\[
\{f_\lambda(x, y)+h_\lambda(x, y)E_\lambda(x, y)\}
- z g_\lambda(x). 
\]
The difference from $f_\lambda(x, y)- z g_\lambda(x)$ is 
\[
\{f_\lambda(x, y)+h_\lambda(x, y)E_\lambda(x, y)\} - 
f_\lambda(x, y) \\
=
 h_\lambda(x, y) E_\lambda(x, y)\in \mathcal{I}, 
\]
since $\mathcal{I}$ defines a subscheme of $\mathcal{C}$. 
Thus the resulting ideals are the same. 
\end{proof}

For points of the curve which are not on the boundary, we consider the following. 
\begin{definition}
Let $\mathcal{E}=\{E_{\lambda}\}_{\lambda\in\Lambda}$ 
be a family of equations over $\Lambda$ 
with $E_0(0, 0)=0$. 

We define $\mathcal{S}'_{\mathcal{E}}$  
to be the closed formal subscheme of $\hat{\mathbb{A}}^3_{\Lambda}$ 
given by 
$\langle z-E_\lambda(x, y)\rangle$, 
and 
\[
L_{\mathcal{E}}^{\prime d}: \mathrm{Hilb}^d(\hat{\mathbb{A}}^2_{\Lambda}/\Lambda)
\to \mathrm{Hilb}^d(\hat{\mathbb{A}}^3) 
\]
to be the composite 
\[
\mathrm{Hilb}^d(\hat{\mathbb{A}}^2_{\Lambda}/\Lambda) 
\overset{\sim}{\to} 
\mathrm{Hilb}^d(\mathcal{S}'_{\mathcal{E}}/\Lambda)
\to
\mathrm{Hilb}^d(\hat{\mathbb{A}}^3). 
\]
If $I\subset\mathcal{O}$ is an ideal of colength $d$ 
and $[\tilde{I}]=L_{\mathcal{E}}^{\prime d}([I])$, 
we write the morphism between the completions as 
\[
L'_{\mathcal{E}, I}: 
\FormalHilb{\hat{\mathbb{A}}^2_{\Lambda}/\Lambda}{I}
\to \FormalHilb{\hat{\mathbb{A}}^3}{\tilde{I}}. 
\]
\end{definition}

The following theorem gives a correspondence of 
our relative Hilbert scheme of family of curves 
and Hilbert schemes of surface singularities. 

\begin{theorem}\label{thm_corr}
Let $\Lambda$ 
be the formal spectrum of an object of $\hat{\mathcal{A}}_{\mathbb{C}}$. 
\begin{itemize}
\item
For $i=1, \dots, k$, 
let $\mathcal{E}_i=\{E_{i, \lambda}\}_{\lambda\in\Lambda}$ 
be a family of $w_i$-contact equations 
and $\mathcal{C}_i$ the family of curves defined by $\mathcal{E}_i$ 
with central fiber $C_i$. 
\item
For $j=1, \dots, l$, 
let $\mathcal{E}'_j=\{E'_{\lambda, j}\}_{\lambda\in\Lambda}$ 
be a family of equations 
and $\mathcal{C}'_j$ the family of curves defined by $\mathcal{E}'_j$ 
with central fiber $C'_j$. 
\end{itemize}
Let $I_i\subset\mathcal{O}$ (resp. $I'_j\subset\mathcal{O}$) 
be the ideal of a length $d_i$ (resp. $d'_j$) subscheme of $C_i$ (resp. $C'_j$) 
and write $[\tilde{I}_i]=L_{\mathcal{E}_i}^{d_i}([I_i])$ 
(resp. $[\tilde{I}'_j]=L_{\mathcal{E}'_j}^{\prime {d'_i}}([I'_j])$). 
We consider the morphism 
\[
L:=\prod_{i=1}^k L_{\mathcal{E}_i, I_i} \times 
\prod_{j=1}^l L'_{\mathcal{E}'_j, I'_j}: 
\FormalHilb{\hat{\mathbb{A}}^2_{\Lambda}/\Lambda}{I_1}
\times_\Lambda \dots \times_\Lambda 
\FormalHilb{\hat{\mathbb{A}}^2_{\Lambda}/\Lambda}{I'_l}
\to 
\FormalHilb{\hat{\mathbb{A}}^3}{\tilde{I}_1}
\times \dots \times 
\FormalHilb{\hat{\mathbb{A}}^3}{\tilde{I}'_l}. 
 \]
\begin{enumerate}
\item
The equality 
\[
L^{-1}\left(
\prod_{i=1}^k \FormalHilb{S_{w_i-1}}{\tilde{I}_i}
\times
\prod_{j=1}^l \FormalHilb{S'}{\tilde{I}'_j}
\right)
= 
\FormalHilb{\mathcal{C}_i/\Lambda}{I_1}
\times_\Lambda \dots \times_\Lambda 
\FormalHilb{\mathcal{C}'_j/\Lambda}{I'_l}
\]
holds as formal subschemes of 
$\FormalHilb{\hat{\mathbb{A}}^2_{\Lambda}/\Lambda}{I_1}
\times_\Lambda \dots \times_\Lambda 
\FormalHilb{\hat{\mathbb{A}}^2_{\Lambda}/\Lambda}{I'_l}$, 
where $S_n\subset \hat{\mathbb{A}}^3$ (resp. $S'\subset \hat{\mathbb{A}}^3$) 
is defined by $yz+x^{n+1}=0$ (resp. $z=0$). 
\item
Assume that $\Lambda$ is regular 
and that the map 
\[
(\Phi_{\mathcal{E}_1, I_1}, \dots, \Delta_{\mathcal{E}'_l, I'_l}): 
\Lambda \to \mathcal{O}/I_1\times\dots\times\mathcal{O}/I'_l 
\]
is surjective. 
Then $L$ is smooth 
of relative dimension $\dim \Lambda-\sum_{i=1}^k d_i - \sum_{j=1}^l d'_j$ 
at $([I_1], \dots, [I'_l])$, 
and therefore induces a smooth morphism 
\[
\FormalHilb{\mathcal{C}_1/\Lambda}{I_1}
\times_\Lambda \dots \times_\Lambda 
\FormalHilb{\mathcal{C}'_l/\Lambda}{I'_l}
\to 
\prod_{i=1}^k \FormalHilb{S_{w_i-1}}{\tilde{I}_i}
\times
\prod_{j=1}^l \FormalHilb{S'}{\tilde{I}'_j}
\]
of relative dimension $\dim \Lambda-\sum_{i=1}^k d_i - \sum_{j=1}^l d'_j$. 
\end{enumerate}
\end{theorem}

\begin{proof}
(1)
Let $R=\mathcal{O}_{\Lambda}$, 
$\mathcal{O}_R=R[[x, y]]$, 
$\tilde{\mathcal{O}}=\mathbb{C}[[x, y, z]]$ 
and $\tilde{\mathcal{O}}_R=R[[x, y, z]]$. 
For an object $A$ of $\mathcal{A}_R$ 
and families $\mathcal{I}_i\subset 
\mathcal{O}_R\otimes_R A\cong\mathcal{O}\otimes A$ (resp. $\mathcal{I}'_j$) 
over $A$ belonging to 
$\FormalHilb{\mathcal{C}_i/\Lambda}{I_i}$ 
(resp. $\FormalHilb{\mathcal{C}'_j/\Lambda}{I'_j}$), 
$\mathcal{I}_i$ 
is mapped to 
$\tilde{\mathcal{I}}_i:=\mathcal{I}_i\cdot\tilde{\mathcal{O}}_R\otimes_R A+ 
(f_{i, \lambda}(x, y) - zg_{i, \lambda}(x))\cdot\tilde{\mathcal{O}}_R\otimes_R A$, 
and $\mathcal{I}'_j$ to 
$\tilde{\mathcal{I}}'_j:=\mathcal{I}'_j\cdot\tilde{\mathcal{O}}_R\otimes_R A+ 
(z-E'_{j, \lambda}(x, y))\cdot\tilde{\mathcal{O}}_R\otimes_R A$. 

The ideal $\tilde{\mathcal{I}}_i$ defines an $A$-valued point of 
$\FormalHilb{S_{w_i-1}}{\tilde{I}_i}$ 
if and only if $yz+x^{w_i}\in \tilde{\mathcal{I}_i}$. 
Since $z\equiv g_{i, \lambda}(x)^{-1}f_{i, \lambda}(x, y)\mod \tilde{\mathcal{I}_i}$, 
this is equivalent to 
$yg_{i, \lambda}(x)^{-1}f_{i, \lambda}(x, y) + x^{w_i}\in \tilde{\mathcal{I}_i}$, 
or $yf_{i, \lambda}(x, y) + x^{w_i}g_{i, \lambda}(x)\in \tilde{\mathcal{I}_i}$. 
Since it is elementary to see that 
$\tilde{\mathcal{I}_i}\cap (\mathcal{O}_R\otimes_R A)=\mathcal{I}_i$, 
this holds if an only if $yf_{i, \lambda}(x, y) +x^{w_i} g_{i, \lambda}(x)\in \mathcal{I}_i$, 
i.\,e. $[\mathcal{I}_i]$ is an $A$-valued point of 
$\FormalHilb{\mathcal{C}_i/\Lambda}{I_i}$. 

Similarly, $\tilde{\mathcal{I}}'_j$ is an $A$-valued point of 
$\FormalHilb{S'}{\tilde{I}'_j}$ 
if and only if $z\in \tilde{\mathcal{I}}'_j$, 
which is equivalent to 
$E'_{j, \lambda}(x, y)\in \tilde{\mathcal{I}}'_j$, 
and to $E'_{j, \lambda}(x, y)\in \mathcal{I}'_j$. 

\medskip
(2)
The space 
$\FormalHilb{\hat{\mathbb{A}}^2_{\Lambda}/\Lambda}{I_1}
\times_\Lambda \dots \times_\Lambda 
\FormalHilb{\hat{\mathbb{A}}^2_{\Lambda}/\Lambda}{I'_l}$
is isomorphic to 
$\FormalHilb{\hat{\mathbb{A}}^2}{I_1}\times\dots\times
\FormalHilb{\hat{\mathbb{A}}^2}{I'_l}\times\Lambda$, 
and hence is smooth over 
$\FormalHilb{\hat{\mathbb{A}}^2}{I_1}\times\dots\times
\FormalHilb{\hat{\mathbb{A}}^2}{I'_l}$
of relative dimension $\dim \Lambda$. 
On the other hand, 
$
\FormalHilb{\hat{\mathbb{A}}^3}{\tilde{I}_1}
\times \dots \times 
\FormalHilb{\hat{\mathbb{A}}^3}{\tilde{I}'_l} 
$
is smooth over 
$\FormalHilb{\hat{\mathbb{A}}^2}{I_1}\times\dots\times
\FormalHilb{\hat{\mathbb{A}}^2}{I'_l}$
of relative dimension $\sum_{i=1}^k d_i + \sum_{j=1}^l d'_j$ 
by Proposition \ref{prop_hilb_and_smooth_morphism}. 

Now $L$ is a morphism over 
$\FormalHilb{\hat{\mathbb{A}}^2}{I_1}\times\dots\times
\FormalHilb{\hat{\mathbb{A}}^2}{I'_l}$
and maps the section 
$\FormalHilb{\hat{\mathbb{A}}^2}{I_1}\times\dots\times
\FormalHilb{\hat{\mathbb{A}}^2}{I'_l}\times\{0\}$ 
to 
$\FormalHilb{(\mathcal{S}_{\mathcal{E}_1})_{0}}{\tilde{I}_1}
\times\dots\times
\FormalHilb{(\mathcal{S}_{\mathcal{E}'_l})_{0}}{\tilde{I}'_l}$. 
The fiber $\{([I_1], \dots, [I'_l])\}\times \Lambda\cong\Lambda$ 
is mapped by $L$ 
in the following way: 
\[
\lambda\mapsto 
([I_1\tilde{\mathcal{O}}+(z-g_{1, \lambda}(x)^{-1}f_{1, \lambda}(x, y))\tilde{\mathcal{O}}], 
\dots, 
[I'_l\tilde{\mathcal{O}}+(z-E'_{l, \lambda}(x, y))\tilde{\mathcal{O}}]). 
\]
The tangent space to the fiber of 
$
\FormalHilb{\hat{\mathbb{A}}^3}{\tilde{I}_1}
\times \dots \times 
\FormalHilb{\hat{\mathbb{A}}^3}{\tilde{I}'_l} 
\to
\FormalHilb{\hat{\mathbb{A}}^2}{I_1}\times\dots\times
\FormalHilb{\hat{\mathbb{A}}^2}{I'_l}$ 
can be identified with 
$\mathcal{O}/I_1\times\dots\times\mathcal{O}/I'_l$, 
and then the tangent map 
$T_0\Lambda \to \mathcal{O}/I_1\times\dots\times\mathcal{O}/I'_l$ 
is nothing but 
$(\Phi_{\mathcal{E}_1, I_1}, \dots, \Delta_{\mathcal{E}'_l, I'_l})$.
Thus we have the assertion. 
\end{proof}

\section{The case of rational log Calabi-Yau surfaces}

We prove the following theorem. 

\begin{theorem}\label{thm_main}
Let $X$ be a smooth projective rational surface, 
$D\subset X$ an anticanonical curve, 
$\beta\in H^2(X, \mathbb{Z})$ a curve class of arithmetic genus $g$ 
and $\mathfrak{d}=\sum_{i=1}^k w_iP_i$ an effective divisor on $D_{\mathrm{sm}}$ 
such that $\beta|_D\sim \mathfrak{d}$. 

Then each point of $\Moduli{X}{D}{\beta}{\mathfrak{d}}$ 
is formally isomorphic to a point 
of $\prod_{i=1}^k \mathrm{Hilb}^{d_i}(S_{w_i-1})\times \mathbb{A}^{2d}$, 
where $S_n$ is a surface singularity of type $A_n$, 
$d_i\leq \delta(C, P_i)$ and $d=g-\sum_{i=1}^k d_i$. 
\end{theorem}

To prove the theorem, 
first we look at the spaces of global sections of certain sheaves 
related to our family of curves. 
\begin{lemma}\label{lem_main_surjectivity}
Let $X$ be a smooth surface, $D\subset X$ an effective divisor 
and $C\subset X$ a proper integral curve. 
\begin{enumerate}
\item
If $X$ is projective, $h^1(\mathcal{O}_X)=0$ and $K_X+D\sim 0$, 
then the natural map $H^0(\mathcal{O}_X(C-D))\to H^0(\mathcal{O}_C(C-D))$ 
is surjective. 
\item
Assume $(K_X+D)|_C\sim 0$ and let $Z\subset C$ be a $0$-dimensional closed subscheme 
such that $\mathrm{Hom}_{\mathcal{O}_C}(I_Z, \mathcal{O}_C)\cong\mathbb{C}$. 
Then the natural map 
$H^0(\mathcal{O}_C(C-D))\to H^0(\mathcal{O}_C(C-D)|_Z)$ 
is surjective. 
\end{enumerate}
\end{lemma}
\begin{proof}
(1)
This follows from the short exact sequence 
$0\to \mathcal{O}_X(-D)\to \mathcal{O}_X(C-D)\to\mathcal{O}_C(C-D)\to 0$ 
and $h^1(\mathcal{O}_X(-D))=h^1(\mathcal{O}_X(K_X+D))=0$. 

\medbreak
(2)
The short exact sequence 
\[
0\to I_Z\cdot\mathcal{O}_C(C-D) \to\mathcal{O}_C(C-D)\to 
\mathcal{O}_C(C-D)|_Z\to 0 
\]
induces the exact sequence 
\[
H^0(\mathcal{O}_C(C-D))\to 
H^0(\mathcal{O}_C(C-D)|_Z)
\to 
H^1(I_Z\cdot\mathcal{O}_C(C-D))\to 
H^1(\mathcal{O}_C(C-D)). 
\]
The last homomorphism is dual to the map 
\[
\mathrm{Hom}_{\mathcal{O}_C}(\mathcal{O}_C, \omega_C(D-C))
\to \mathrm{Hom}_{\mathcal{O}_C}(I_Z, \omega_C(D-C)) 
\]
induced by $I_Z\hookrightarrow \mathcal{O}_C$. 
Since $\omega_C(D-C)\cong \mathcal{O}_C(K_X+D)\cong \mathcal{O}_C$, 
this map is an isomorphism by the assumption, 
hence the assertion. 
\end{proof}

Next, we relate these maps 
to the maps $\Phi_{\mathcal{E}, I}$ introduced in 
Definition \ref{def_Phi_nondegeneracy}. 

Let $X$, $D$, $\beta$ and $\mathfrak{d}$ 
be as in the theorem, 
$\Lambda_{\beta, \mathfrak{d}}$ as in \eqref{eq_lambda_b_d} 
and $[C]\in\Lambda_{\beta, \mathfrak{d}}$. 
As usual, for a divisor $A$ on $X$, 
we regard $\mathcal{O}_X(A)$ as 
the subsheaf of the constant sheaf $K(X)$ of rational functions 
defined by $\mathcal{O}_X(A)(U)=\{f\in K(X): (\mathrm{div}(f)+A)|_U\geq 0\}$. 
Let $\sigma_A$ denote $1\in K(X)$ regarded as a rational section of $\mathcal{O}_X(A)$. 
If $A$ is effective, 
$\sigma_A$ is a global section defining the divisor $A$. 

Let $0\to H^0(\mathcal{O}_X(C-D))\overset{\alpha}{\to} 
H^0(\mathcal{O}_X(C))\to
H^0(\mathcal{O}_X(C)|_D)$ 
be the natural exact sequence. 
By Lemma \ref{lem_linearsystem} (1), 
$\Lambda_{\beta, \mathfrak{d}}$ is an open subset of  
$\mathbb{P}_*(\mathbb{C}\sigma_C+\mathrm{Im}\ \alpha)\subset
\mathbb{P}_*(H^0(\mathcal{O}_X(C)))$. 
The affine subspace $\sigma_C+\mathrm{Im}\ \alpha\subset
\mathbb{C}\sigma_C+\mathrm{Im}\ \alpha$ 
can be identified with an affine open neighborhood of $[C]$ 
in $\mathbb{P}_*(\mathbb{C}\sigma_C+\mathrm{Im}\ \alpha)$, 
and this gives a local isomorphism 
$([C]\in\Lambda_{\beta, \mathfrak{d}})\cong (0\in H^0(\mathcal{O}_X(C-D)))$ 
and an identification 
$T_{[C]}\Lambda_{\beta, \mathfrak{d}}\cong H^0(\mathcal{O}_X(C-D))$. 

For a point $P\in C\cap D$, let $y\in\mathcal{O}_{X, P}$ be a local equation for $D$, 
$x\in\mathcal{O}_{X, P}$ be such that $(x, y)$ is a local parameter system at $P$. 
Then we can take a local equation $E_0^P\in\mathbb{C}[[x, y]]$ for $C$ 
of the form $E_0^P=yf_0^P+x^{w_P}$, 
where $w_P=(C.D)_P$. 

Now $\rho_D:=y^{-1}\sigma_D$, $\rho_C:=(E_0^P)^{-1}\sigma_C$ 
and $\rho_{C-D}:=y\cdot(E_0^P)^{-1}\sigma_{C-D}$ 
are local generators of $\mathcal{O}_X(D)$, $\mathcal{O}_X(C)$ 
and $\mathcal{O}_X(C-D)$ at $P$, respectively.

\begin{lemma}\label{lem_main_localdesc}
\begin{enumerate}
\item
Under the local isomorphism 
$([C]\in\Lambda_{\beta, \mathfrak{d}})\cong (0\in H^0(\mathcal{O}_X(C-D)))$, 
the tautological family over $\Lambda_{\beta, \mathfrak{d}}$ is given 
locally at $P$ by the family of equations 
$\mathcal{E}^{P}:=\{E_0^P + yf_\lambda^P\}_{\lambda\in U}$, 
where $U\subseteq H^0(\mathcal{O}_X(C-D))$ is a neighborhood of $0$ 
and $f_\lambda^P\in\mathcal{O}_{X, P}$ is defined by 
$\lambda=f_\lambda^P\cdot\rho_{C-D}$. 
\item
Let $Z$ be a $0$-dimensional subscheme of $C$ supported at $P$. 
Under the identifications 
$T_{[C]}\Lambda_{\beta, \mathfrak{d}}\cong H^0(\mathcal{O}_X(C-D))$ 
and 
\[
H^0(\mathcal{O}_C(C-D)|_Z)
\overset{\sim}{\to} H^0(\mathcal{O}_Z);\ (f\cdot\rho_{C-D})|_Z\mapsto f|_Z, 
\]
the natural map 
$
H^0(\mathcal{O}_X(C-D))\to 
H^0(\mathcal{O}_C(C-D)|_Z)
$
coincides with $\Phi_{\mathcal{E}^P, Z}$. 
\end{enumerate}
\end{lemma}

\begin{proof}
(1)
For a section $\lambda\in H^0(\mathcal{O}_X(C-D))$ 
which is locally written as $f_\lambda^P\cdot \rho_{C-D}$, 
its image under the identification 
$H^0(\mathcal{O}_X(C-D))\overset{\sim}{\to} \sigma_C+\mathrm{Im}\ \alpha$ 
is 
\[
\sigma_C+f_\lambda^P\cdot \rho_{C-D}=\sigma_C+f_\lambda^P\cdot y\cdot(E_0^P)^{-1}\sigma_{C-D}
= 1+yf_\lambda^P\cdot(E_0^P)^{-1}=(E_0^P+yf_\lambda^P)\rho_C, 
\]
hence the assertion. 

\medbreak
(2)
Let us take a basis $\{f_i\cdot \rho_{C-D}\}_{i=1}^g$ of $H^0(\mathcal{O}_X(C-D))$ 
and introduce a coordinate system $(s_1, \dots, s_g)$ 
to write an element of $H^0(\mathcal{O}_X(C-D))$ 
as $f_{s_1, \dots, s_g}^P=(\sum_{i=1}^g s_if_i)\cdot \rho_{C-D}$. 
We can also think of $(s_1, \dots, s_g)$ as coordinates on 
$\Lambda_{\beta, \mathfrak{d}}$ at $[C]$, 
and under the identification 
$H^0(\mathcal{O}_X(C-D))\overset{\sim}{\to}T_{[C]}\Lambda_{\beta, \mathfrak{d}}$, 
the element $f_i\cdot \rho_{C-D}$ 
corresponds to the tangent vector $\partial/\partial s_i$. 
As above, we have
\[
\sigma_C + f_{s_1, \dots, s_g}^P\cdot \rho_{C-D}
= (E_0^P+yf_{s_1, \dots, s_g}^P)\rho_C
= \left(y\left(f_0^P+\sum_{i=1}^g s_if_i\right)+x^{w_P}\right)\rho_C, 
\]
so $\Phi_{\mathcal{E}^P, Z}$ maps $\partial/\partial s_i$ to $f_i|_Z$. 
This coincides with 
the image of $f_i\cdot \rho_{C-D}$ 
by the natural map 
$
H^0(\mathcal{O}_X(C-D))\to 
H^0(\mathcal{O}_C(C-D)|_Z). 
$ 
\end{proof}

\begin{proof}[Proof of Theorem \ref{thm_main}]
Let $\mathcal{F}$ be a sheaf on $X$ 
represented by a point of $\Moduli{X}{D}{\beta}{\mathfrak{d}}$. 
Let $C$ be the support of $\mathcal{F}$. 
Then, by Proposition \ref{prop_twisted_abel_map}, 
there exists a $0$-dimensional subscheme $Z$ of $C$ 
such that $([\mathcal{F}]\in\Moduli{X}{D}{\beta}{\mathfrak{d}})$ is 
formally isomorphic to 
$([Z]\in \mathrm{Hilb}^g(\mathcal{C}_{\beta, \mathfrak{d}}/\Lambda_{\beta, \mathfrak{d}}))$ 
and that $\mathrm{Hom}_{\mathcal{O}_C}(I_Z, \mathcal{O}_C)\cong \mathbb{C}$. 

Let $\{P_1, \dots, P_k\}=\mathrm{Supp}\ Z\cap D$ 
and  $\{Q_1, \dots, Q_l\}=\mathrm{Supp}\ Z\setminus D$. 
By taking local coordinates and choosing local equations for $D$ and $C$ 
at each of these points, 
we have families of equations $\mathcal{E}^{P_i}$ and $\mathcal{E}^{Q_j}$, 
and a map 
\[
(\Phi_{\mathcal{E}^{P_1}, Z_{P_1}}, \dots, \Delta_{\mathcal{E}^{Q_l}, Z_{Q_l}})
:
T_{[C]}\Lambda_{\beta, \mathfrak{d}}
\to
\prod_{i=1}^k \mathcal{O}_{Z_{P_i}}
\times
\prod_{i=1}^l \mathcal{O}_{Z_{Q_i}}, 
\]
where $Z_P$ denotes the connected component of $Z$ supported at $P$. 
By Lemma \ref{lem_main_localdesc}, 
this can be identified with 
the natural map $H^0(\mathcal{O}_X(C-D))\to H^0(\mathcal{O}_C(C-D)|_Z)$, 
and it is surjective by Lemma \ref{lem_main_surjectivity}. 

By Theorem \ref{thm_corr}, 
we see that 
$[Z]\in \mathrm{Hilb}^g(\mathcal{C}_{\beta, \mathfrak{d}}/\Lambda_{\beta, \mathfrak{d}})$ 
is formally isomorphic to 
a point of $\prod_{i=1}^k \mathrm{Hilb}^{d_i} S_{w_i-1}\times \mathbb{A}^{2d}$, 
where $d_i$ is the length of $Z$ at $P_i$ and $d=g-\sum_{i=1}^k d_i$. 
Since $Z$ can be chosen so that $Z_{P_i}$ is contained 
in the conductor of $C$ by Proposition \ref{prop_twisted_abel_map}, 
we may assume that $d_i\leq \delta(C, P_i)$. 
\end{proof}

\section*{Acknowledgements}
The author would like to thank Akira Ishii and Daisuke Matsushita 
for answering questions on symplectic singularities and Hilbert schemes. 
This work was supported by JSPS KAKENHI Grant Number JP22K03229.


\begin{thebibliography}{99}

\bibitem{AK1980}
A.\ B.\ Altman, S.\ L.\ Kleiman, 
\textit{Compactifying the Picard Scheme}, 
Adv. Math. \textbf{35} (1980), 50--112. 

\bibitem{AS2018}
E.\ Arbarello, G.\ Sacc\`{a}, 
\textit{Singularities of moduli spaces of sheaves on K3 surfaces 
and Nakajima quiver varieties}, 
Adv. Math. \textbf{329} (2018), 649--703.

\bibitem{Beauville1990}
A.\ Beauville, 
\textit{Jacobiennes des courbes spectrales et syst\`emes hamiltoniens 
compl\`etement int\'egrables}, 
Acta Math. \textbf{164} (1990), 211--235. 

\bibitem{BiswasGomez2020}
I.\ Biswas, T.\ L.\ G\'omez, 
\textit{Poisson structure on the moduli spaces of sheaves
of pure dimension one on a surface}, 
Geom. Dedicata \textbf{207} (2020), 157--165. 

\bibitem{Bottacin1995a}
F.\ Bottacin, 
\textit{Symplectic geometry on moduli spaces of stable pairs}, 
Ann. Sci. \'{E}cole Norm. Sup. (4) \textbf{28} (1995), no. 4, 391--433.

\bibitem{Bottacin1995b}
F.\ Bottacin, 
\textit{Poisson structure of moduli spaces of sheaves
over Poisson surfaces}, 
Invent. Math. \textbf{121} (1995), 421--436. 

\bibitem{CGKT2021a}
J.\ Choi, M.\ van Garrel, S.\ Katz, N.\ Takahashi, 
\textit{Log BPS numbers of log Calabi-Yau surfaces}, 
Trans. Amer. Math. Soc. \textbf{374} (2021), no. 1, 687--732.

\bibitem{CGKT2021b}
J.\ Choi, M.\ van Garrel, S.\ Katz, N.\ Takahashi, 
\textit{Sheaves of maximal intersection and multiplicities of stable log maps}, 
Selecta Math. (N.S.) \textbf{27} (2021), no. 4, Paper No. 61. 

\bibitem{CKLR2024}
A.\ Czapli\'{n}ski, A.\ Krug, M.\ Lehn, S.\ Rollenske, 
\textit{Compactified Jacobians of extended ADE curves and Lagrangian fibrations}, 
Commun. Contemp. Math. \textbf{26} (2024), no. 10, Paper No. 2450004, 46 pp.

\bibitem{CY2023}
A.\ Craw, R.\ Yamagishi, 
\textit{The Le Bruyn-Procesi theorem and Hilbert schemes}, 
arXiv:2312.08527. 

\bibitem{Das2022}
S.\ Das, 
\textit{Relative log-symplectic structure on a semi-stable degeneration of moduli of Higgs bundles}, 
Adv. Math. \textbf{410} (2022), Paper No. 108756, 61 pp.

\bibitem{DSouza1979}
C.\ D'Souza, 
\textit{Compactification of generalised Jacobians}, 
Proc. Indian Acad. Sci. Sect. A Math. Sci. \textbf{88} (1979), no. 5, 419--457.

\bibitem{FGS1999}
B.\ Fantechi, L.\ G\"ottsche, D.\ van\ Straten, 
\textit{Euler number of the compactified Jacobian 
and multiplicity of rational curves}, 
J. Alg. Geom. \textbf{8} (1999), 115--133. 

\bibitem{Hartshorne1986}
R.\ Hartshorne, 
\textit{Generalized divisors on Gorenstein curves and a theorem of Noether}, 
J. Math. Kyoto Univ. \textbf{26} (1986), no. 3, 375--386. 

\bibitem{KLS2006}
D.\ Kaledin, M.\ Lehn, Ch.\ Sorger, 
\textit{Singular symplectic moduli spaces}, 
Invent. Math. \textbf{164} (2006), no. 3, 591--614.

\bibitem{LW2015}
J.\ Li, B.\ Wu, 
\textit{Good degeneration of Quot-schemes and coherent systems}, 
Comm. Anal. Geom. \textbf{23} (2015), no. 4, 841--921.

\bibitem{Markman1994}
E.\ Markman, 
\textit{Spectral curves and integrable systems}, 
Compos. Math. \textbf{93} (1994), 255--290. 

\bibitem{MPT2010}
D.\ Maulik, R.\ Pandharipande, R.\ P.\ Thomas, 
\textit{Curves on K3 surfaces and modular forms}, 
With an appendix by A. Pixton, 
J. Topol. \textbf{3} (2010), no. 4, 937--996.

\bibitem{MRV2017}
M.\ Melo, A.\ Rapagnetta, F.\ Viviani, 
\textit{Fine compactified Jacobians of reduced curves}, 
Trans. Amer. Math. Soc. \textbf{369} (2017), no. 8, 5341--5402.

\bibitem{Mukai1984}
S.\ Mukai, 
\textit{Symplectic structure of the moduli space of sheaves 
on an abelian or K3 surface}, 
Invent. Math. \textbf{77} (1984), 101--116.

\bibitem{Nitsure1991}
N.\ Nitsure, 
\textit{Moduli space of semistable pairs on a curve}, 
Proc. London Math. Soc. \textbf{62} (1991), 275--300. 

\bibitem{OG1999}
K.\ G.\ O'Grady, 
\textit{Desingularized moduli spaces of sheaves on a K3}, 
J. Reine Angew. Math. \textbf{512} (1999), 49--117.

\bibitem{PR2023}
A.\ Perego, A.\ Rapagnetta, 
\textit{Irreducible symplectic varieties from moduli spaces of sheaves 
on K3 and Abelian surfaces}, 
Algebr. Geom. \textbf{10} (2023), no. 3, 348--393.

\bibitem{Piene1978}
R.\ Piene, 
\textit{Polar classes of singular varieties}, 
Ann. Sci. \'{E}cole Norm. Sup. (4) \textbf{11} (1978), no. 2, 247--276.

\bibitem{Sernesi2006}
E.\ Sernesi, 
Deformations of algebraic schemes, 
Grundlehren Math. Wiss. \textbf{334}, Fundamental Principles of Mathematical Sciences, 
Springer-Verlag, Berlin, 2006. xii+339 pp.

\bibitem{Shende2012}
V.\ Shende, 
\textit{Hilbert schemes of points on a locally planar curves 
and the Severi strata of its versal deformation}, 
Compos. Math. \textbf{148} (2012), 531--547. 

\bibitem{Tyurin1988}
A.\ N.\ Tyurin, 
\textit{Symplectic structures on the moduli spaces of vector bundles on algebraic surfaces with $p_g>0$} (Russian), 
Izv. Akad. Nauk SSSR Ser. Mat. \textbf{52} (1988), no. 4, 813--852; 
translation in Math. USSR-Izv. \textbf{33} (1989), no. 1, 139--177.


\end{thebibliography}
\end{document}